\documentclass{article}

\usepackage[utf8]{inputenc}
\usepackage{amssymb, amsmath, amsthm}
\usepackage{subfig}
\usepackage{algorithm, algpseudocode}
\usepackage{tikz}
\usepackage{pgfplots}
\usepackage{pgfplotstable}
\usepackage{csvsimple}
\usepackage{hyperref}

\newtheorem{lemma}{Lemma}[section]

\newtheorem{proposition}[lemma]{Proposition}

\theoremstyle{remark}

\DeclareMathOperator*{\argmin}{argmin}

\newcommand{\real}{\mathbb{R}}

\newcommand{\interpol}{\mathcal{I}}
\newcommand{\transforms}{\mathcal{T}}

\newcommand{\param}{\mathcal{P}}

\newcommand{\bmu}{{\bar{\mu}}}
\newcommand{\hmu}{{\hat{\mu}}}
\newcommand{\heta}{{\hat{\eta}}}

\newcommand{\tsi}{\mathcal{TI}}
\newcommand{\tsihp}{\mathcal{T}_{hp}\mathcal{I}}
\newcommand{\grid}{\mathcal{T}}
\DeclareMathOperator*{\id}{Id}
\newcommand{\muSingular}{\bar{\mu}}

\begin{document}

\title{$h$ and $hp$-adaptive Interpolation by Transformed Snapshots for Parametric and Stochastic Hyperbolic PDEs}

\author{G. Welper\footnote{Department of Mathematics, University of Southern California, Los Angeles, CA 90089, USA, email \href{mailto:welper@usc.edu}{\texttt{welper@usc.edu}}}}
\date{}

\maketitle

\begin{abstract}

The numerical approximation of solutions of parametric or stochastic hyperbolic PDEs is still a serious challenge. Because of shock singularities, most methods from the elliptic and parabolic regime, such as reduced basis methods, POD or polynomial chaos expansions, show a poor performance. Recently, Welper [Interpolation of functions with parameter dependent jumps
by transformed snapshots. SIAM Journal on Scientific Computing,
39(4):A1225–A1250, 2017] introduced a new approximation method, based on the alignment of the jump sets of the snapshots. If the structure of the jump sets changes with parameter, this assumption is too restrictive. However, these changes are typically local in parameter space, so that in this paper, we explore $h$ and $hp$-adaptive methods to resolve them. Since local refinements do not scale to high dimensions, we introduce an alternative ``tensorized'' adaption method. 
 
\end{abstract}



\section{Introduction}

During the last decade numerical methods for parametric and stochastic PDEs have evolved into a large research field. This lead to many methods for elliptic and parabolic problems such as the reduced basis method \cite{RozzaHuynhPatera2008,SenVeroyHuynhEtAl2006,PateraRozza2006}, Proper Orthogonal Decompositions (POD) \cite{Sirovich1987,KunischVolkwein2001,KunischVolkwein2002}, stochastic Galerkin and stochastic collocation methods \cite{XiuKarniadakis2002,BabuskaTemponeZouraris2004,BabuskaNobileTempone2007,GunzburgerWebsterZhang2014} among others. In comparison, the literature on the hyperbolic case is comparatively scarce. Stability issues and shocks cause significant difficulties and require fundamentally new concepts beyond the known methods.

One of the major obstructions is the prevalence of shock discontinuities in the solutions $u(x,\mu)$ of parametric hyperbolic problems, with physical variables $x \in \Omega \subset \real^m$ (possibly containing time) and parameters $\mu \in \param \subset \real^d$ (possibly random). It has been shown \cite{Welper2015,OhlbergerRave2016}
that already for simple examples no \emph{separation of variables} approach of the type
\begin{equation}
u(x,\mu) \approx u_n(x, \mu) := \sum_{i=1}^n c_i(\mu) \psi_i(x),
\label{eq:polyadic-decomposition}
\end{equation}
can achieve errors $\max_{\mu \in \param} \|u(\cdot, \mu) - u_n(
\cdot, \mu)\|_{L_1(\Omega)}$  with rates higher than $\mathcal{O}(n)$. Practical issues for more realistic problems can be found e.g. in \cite{IaccarinoPetterssonNordstroemWitteveen2010}. The choice of norms is natural in the context of reduced basis methods but other $L_p$ norms, like weighted $L_2$-norms in a stochastic setting, do not change the underlying issues. Note that most of the established methods such as reduced basis, POD or polynomial chaos expansions are based on separation of variables with different choices for $c_i$ and $\psi_i$. Thus none of these methods is expected to yield good reconstructions out of the box.

There is comparatively little work in the literature that addresses these problems. Many papers on parametric hyperbolic problems do use separation of variables \eqref{eq:polyadic-decomposition} and focus on different problems like solving the PDE directly in a reduced basis, stabilization, online/offline decompositions and error estimators, see e.g. \cite{ChenGottliebHesthaven2005,HaasdonkOhlberger2008,HaasdonkOhlberger2008a,NguyenRozzaPatera2009,PetterssonAbbasbIaccarinoEtAl2010,PulchXiu2012,TryoenMaitreErn2012,MishraSchwab2012,DespresPoeetteLucor2013,YanoPateraUrban2014,PetterssonIaccarinoNordstroem2014,PacciariniRozza2014,DahmenPleskenWelper2014,Dahmen2015, AbgrallAmsallem2015,JinXiuZhu2016}. Only more recently some authors directly addressed approximation problems, \cite{ConstantineIaccarino2012,GerbeauLombardi2012,OhlbergerRave2013,JakemanNarayanXiu2013,GerbeauLombardi2014,TaddeiPerottoQuarteroni2015,Welper2015,ReissSchulzeSesterhenn2015,CagniartMadayStamm2016,RimMoeLeVeque2017}. 

This paper is based on the approach in \cite{Welper2015}, where an additional transform of the physical domain is introduced in order to align the jump discontinuities in parameter. Specifically, for interpolation points $\param_n \subset \param$, corresponding snapshots $u(\cdot, \eta)$, $\eta \in \param_n$ and transforms $(x, \mu) \to \phi(\mu, \eta)(x)$, the solution $u(x, \mu)$ is approximated by the \emph{transformed snapshot interpolation} (TSI)
\begin{equation*}
  u(x,\mu) \approx \sum_{\eta \in \param_n} \ell_\eta(\mu) u(\phi(\mu, \eta)(x), \eta),
\end{equation*}
where $\ell_i$ are Lagrange interpolation polynomials with regard to the interpolation points $\param_n$. First note that in case the transform $\phi(\mu, \eta)(x) = x$ is the identity, this is a standard polynomial interpolation of $\mu \to u(\cdot, \mu)$. The purpose of the transform $\phi(\mu, \eta)(x)$ is to align the jumps of the snapshots at parameter $\eta \in \param_n$ with the jumps of the target function function $u(\cdot, \mu)$. In order to keep the article self contained, this transformed interpolation and the computation of the transforms are summarized in Section \ref{sec:tsi}.

In \cite{Welper2015}, it has been demonstrated that this modified interpolation can effectively deal with jump discontinuities arising form many parametric hyperbolic problems. However, the assumption that the jumps can be aligned is too restrictive for many practical purposes. To be more specific, by aligning transforms, we mean that for any pair of parameters $\mu$ and $\eta$, it must be possible to find a transform that maps the jump set of $u(\cdot, \mu)$ to the jump set of $u(\cdot, \eta)$. With slight regularity assumptions on the transform, this implies that for all parameters the jump sets must be homeomorphic. In the following, we therefore loosely refer to problems where the jump sets cannot be aligned as problems with \emph{changing jump set/shock topology}. Prototype examples are the formation or collision of shocks which changes their number and thus also their topology.
 
Thus, we have have to deal with two types of singularities: shock movement and shock topology changes. The former ones typically cause singularities at every parameter value and are therefore global in nature. They can be dealt with by the TSI. In contrast, the latter topology changes are local in parameter space. Indeed, we can partition the parameter space into patches that allow an alignment of the jump sets so that the topology changes are located at the boundaries of these patches. For low to moderate dimensional parameter spaces one can therefore use classical localization techniques to resolve the shock topology changes. One choice that is perused in this article is to apply the TSI on $h$ or $hp$ adaptively refined cells in parameter space.

This strategy has two drawbacks: First, local refinements do not translate well to high dimensional parameter spaces and second their refined cells have typically linear faces that limit the achievable approximation rate of the boundaries of the alignable patches. Given these difficulties, in this paper, we propose an alternative route via a construction akin to a tensor product. This only makes sense if the local refinements are aligned with the coordinate axis, which is automatically taken care of by suitable transforms in the TSI. Although this approach is certainly susceptible to the curse of dimensionality, many methods that do work in high dimensions start from a tensor product construction. Examples include sparse grids \cite{BungartzGriebel2004} and sparse polynomial expansions \cite{CohenDeVoreSchwab2010,CohenDeVoreSchwab2011,CohenDeVore2015}. To what extend such methods can be applied to ``tensorized'' $h$-adaptive TSI in this paper is left for future research.

The method of this paper is not meant to deal with arbitrary parametric hyperbolic problems. Indeed their shock structure and topology changes can be fairly numerous and complicated. Instead, we focus on a limited number of shock topology changes, only. This already has a number of applications, e.g. for steady state problems that often have a small number of shock configurations. E.g. for an airfoil, we have a normal shock in the transonic regime and a bow and tail shock in the supersonic regime. In the current literature, even this simplified problem is mostly unsolved.

In Section \ref{sec:tsi}, we shortly review the TSI, in Section \ref{sec:hp-adaptivity}, we start out with a simple $hp$-type adaption for one parameter dimension, and Section \ref{sec:multiple-param-dim}, we consider tensor product type constructions. Finally, in Section \ref{sec:experiments} we consider some numerical experiments.

\section{Interpolation by Transformed Snapshots}
\label{sec:tsi}

In order to keep this article self-contained, this section provides an overview of the transformed snapshot interpolation, see \cite{Welper2015} for a more details. To this end, assume we have selected some interpolation points $\param_n$ of cardinality $n$ and computed the corresponding snapshots $u(\cdot, \eta)$, $\eta \in \param_n$. This typically involves a PDE solve with fixed parameter $\eta$, which we assume to be sufficiently accurate in the following. Now, for some new parameter $\mu$ that we have not seen yet, we want to calculate $u(\cdot, \mu)$, ideally without spending the cost necessary to solve the parametric PDE. This problem is exemplified in Figure \ref{fig:tsi-motivation}: It shows the $\Omega \times \param$ plane and the dotted lines indicate the $(x,\mu)$ points where we know the function values $u(x, \mu)$, whereas the dashed line indicates an example for a target parameter, where we want to know the solution $u(\cdot, \mu)$. 

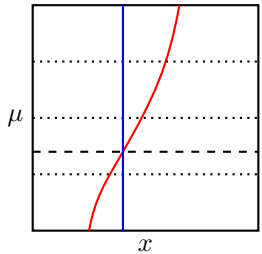
\begin{figure}[htb]

	\centering
	\begin{tikzpicture}[scale=1.5]

	\draw[thick] (0,0) rectangle (2,2);

	\node[below] at (1,0) {$x$};
	\node[left] at (0,1) {$\mu$};

	\draw[thick, dotted] (0, 0.5) -- (2, 0.5);
	\draw[thick, dotted] (0, 1.0) -- (2, 1.0);
	\draw[thick, dotted] (0, 1.5) -- (2, 1.5);
	\draw[thick, dashed] (0, 0.7) -- (2, 0.7);

	\draw[thick, red] (0.5, 0) to[out=80, in=-120] (0.8, 0.7) to[out=-120+180, in=-100] (1.3, 2);
	\draw[thick, blue] (0.8, 0) -- (0.8, 2);

	\end{tikzpicture}

	\caption{Dotted lines: Locations where we know the values $u(x,\mu)$. Dashed line: Locations where we want to approximate $u(x, \mu)$. Red line: Locations of a jump discontinuity. Blue Line: Locations of a jump discontinuity after transformation.}
	\label{fig:tsi-motivation}
\end{figure}

Many numerical schemes are available to solve this problem such as reduced basis methods, empirical interpolation or polynomial chaos expansions. For the sake of simplicity, let us consider the perhaps easiest method which is polynomial interpolation: For every fixed $x$, we can interpolate in the $\mu$ direction, which corresponds to interpolating along vertical lines in the figure with interpolation points on the dotted lines. It will be useful to view this from a different angle. We can consider $u(x, \mu)$ as a function
\begin{align}
u: \param & \to L_1(\Omega), & u_\mu & := u(\cdot, \mu)
\label{eq:banach-space-valued}
\end{align}
so that our interpolation procedure is a polynomial interpolation of Banach-space valued functions. Whatever perspective one takes, for interpolation points $\eta \in \param_n$, and corresponding Lagrange polynomials $\ell_\eta$, the interpolation is given by
\begin{equation}
(\interpol_n u)(x,\mu) := \sum_{\eta \in \param_n}^n \ell_\eta(\mu) u(x, \eta).
\label{eq:interpol}
\end{equation}
This is a reasonable approach in case $u$ is smooth in $\mu$, but in the hyperbolic regime we are explicitly interested in the case where $u$ has jumps. To this end, let us assume that $u(x, \mu)$ has a jump along the red line, in Figure \ref{fig:tsi-motivation}. If we interpolate along vertical lines, there are always $x$-values for which we cross the jump which spoils the accuracy of the interpolation. This problem is not restricted to the perhaps overly simplistic choice of polynomial interpolation. In fact, in \cite{Welper2015,OhlbergerRave2016} it has been shown via a Kolmogorov $n$-with argument that all separation of variable type methods that decompose the solution in the form $\sum_i \psi_i(x) \phi_i(\mu)$ with any choice of $\psi_i$ and $\phi_i$ have similarly low convergence rates. This contains all of the more sophisticated methods mentioned above.

The cause of these difficulties is that no matter what linear combination of the snapshots we choose, the jumps will always be in the wrong places. We therefore introduce a transform in order to align them. That is, if we want to use the snapshot $u(\cdot, \eta)$, $\eta \in \param_n$ to approximate $u(\cdot, \mu)$, we first transform the physical variables $x$ in dependence on the source parameter $\eta$ and target parameter $\mu$ by a transform $\phi(\mu, \eta): \Omega \to \Omega$ leading to the \emph{transformed snapshots}
\begin{equation}
  v_\mu(x, \eta) := v_\mu^\phi(x, \eta) := u(\phi(\mu, \eta)(x), \eta).
  \label{eq:transformed-snapshots}
\end{equation}
We usually omit the superscript $\phi$ when it is known from context. In Figure \ref{fig:tsi-motivation}, this transform moves points $(x, \mu)$ horizontally, including the jump location. Ideally, this aligns the jump locations of the transformed snapshots $v_\mu(x, \eta)$ so that they are in the location of the blue vertical line. Let us assume that we already know such a transform and discuss its calculation below. Then we know the transformed snapshots on the dotted lines, just as the original snapshots. Obviously, for $\eta = \mu$ nothing has to be aligned so that we assume
\begin{equation}
\phi(\mu, \mu)(x) = x,
\label{eq:interpolation-condition}
\end{equation}
which implies that $u(x, \mu) = v_\mu(x, \mu).$ This means that on the dashed line, where we want to reconstruct $u(x, \mu)$, we find that the original snapshot and the transformed snapshots coincide. Thus, instead of interpolating $u(x,\mu)$ along the parameter direction, at the target parameter $\mu$, we can interpolate the transformed snapshot $v_\mu(x, \eta)$ along $\eta$, instead, i.e.
\begin{equation}
u(x,\mu) = v_\mu(x, \mu) \approx (\interpol_n v_\mu)(x, \mu) = \sum_{\eta \in \param_n}^n \ell_\eta(\mu) v_\mu(x, \eta).
\label{eq:tsi}
\end{equation}
In Figure \ref{fig:tsi-motivation}, this is again an interpolation along vertical lines. Since the blue line containing the jump locations of the transformed snapshots is vertical, this interpolation process never encounters the jumps and yields highly accurate approximations.

We still need a practical method to calculate the inner transformations $\phi(\mu, \eta)(x)$. To this end, we choose the transform that minimizes the worst case error
\begin{equation}
\phi = \argmin_{\phi \in \transforms} \, \sup_{\mu \in \param_T} \|u(\cdot, \mu) - (\interpol_n v_\mu^\phi)(\cdot, \mu)\|_{L_1},
\label{eq:opt}
\end{equation}
where $\transforms$ is a set of admissible transforms that satisfies \eqref{eq:interpolation-condition}. This approach is comparable to back propagation learning algorithms for neural networks or greedy methods for reduced basis methods. Ideally, we would use the supremum of all parameters, but since this requires knowledge of $u(\cdot, \mu)$ for all parameters, we confine ourselves to a set of training snapshots $u(\cdot, \mu_t)$ for $\mu_t$ in a training sample $\param_T$. 

For the practical computation of the transformed snapshot interpolation, the optimization problem \eqref{eq:opt} seems forbiddingly complicated. In fact, already for easy problems one has local minima with unacceptable alignment of the jump discontinuities. However, one can greatly alleviate the problem by observing that the transforms $\phi(\mu, \eta)(x)$ allow some additional structure: One can decompose them into local contributions. If we have a chain of parameters $\eta = \mu_0, \dots, \mu_k = \mu$, we can split the transform by
\begin{equation}
  \phi(\mu, \eta)(x) = \phi(\mu_1, \mu_0) \circ \dots \circ \phi(\mu_k, \mu_{k-1})
  \label{eq:transform-decomposition}
\end{equation}
or equivalently 
\begin{equation*}
  \phi(\mu, \eta)^{-1}(x) = \phi(\mu_k, \mu_{k-1})^{-1} \circ \dots \circ \phi(\mu_1, \mu_0)^{-1}
\end{equation*}
which has a more natural ordering of the parameters in case the transform is invertible. One can directly verify that if every local in parameter transform $\phi(\mu_{i+1}, \mu_i)$ aligns the jump discontinuities correctly, than their composition does so as well. For $\mu_{i+1}$ and $\mu_i$ sufficiently close, or rather their jump sets sufficiently close, the transforms $\phi(\mu_{i+1}, \mu_i)$ are perturbations of the identity. Therefore, the identity transform is a good enough initial value for optimizing a local variant of \eqref{eq:opt}, which leads to acceptable alignment for the local problem and thus via \eqref{eq:transform-decomposition} also for the global problem. This construction can be compared to implicit ODE solvers were for each step we have to solve a nonlinear system of equations. They can be solved uniquely if the step-size is small enough, because we have a very localized problem. Indeed, $\phi(\mu, \eta)(x)$ can be understood as a propagator of an ODE in the variable $\eta$ with initial condition given by $x$, see Lemma \ref{lemma:ode} below. In \cite{Welper2015} it is shown that for piecewise continuous one dimensional (with respect to $x$) problems this ``localization'' strategy provably prevents optimizers to be trapped in local minima.

We now define the \emph{transformed snapshot interpolation} (TSI) by
\begin{equation}
\tsi_n u := \interpol_n v_\mu^\phi
\label{eq:tsi-operator}
\end{equation}
with $\phi$ given by \eqref{eq:opt}. This TSI uses the same snapshots as the classical interpolation \eqref{eq:interpol} for the reconstruction with additional training samples for the optimization of the inner transform. In \cite{Welper2015} it has been numerically observed that for one parameter dimension about $n$ training snapshots are sufficient for the optimizer so that an application of $\tsi_n$ needs about $2n$ snapshots. We will later break higher parameter spaces down to this one dimensional case.

Finally, let us summarize the computational steps. Similar to reduced basis methods, the work is split into an offline and online phase.

\paragraph{Offline}
\begin{enumerate}
	\item Compute the snapshots $u(x,\eta)$, $\eta \in \param_n$ by solving the hyperbolic PDE with fixed parameter $\eta$.
	\item For all $\eta \in \param_n$, compute the inner transforms $(x, \mu) \to \phi(\mu, \eta)(x)$ by localized optimization problems of the type \eqref{eq:opt}. Note that typically the transforms are smooth in $x$ and $\mu$ so that that they can be approximated by more classical techniques, unlike $u$ itself.
\end{enumerate}

\paragraph{Online}
\begin{enumerate}
	\item For each new $\mu$, compute an approximation of $u(\cdot, \mu)$ by \eqref{eq:tsi}.
\end{enumerate}

\section{\texorpdfstring{$hp$}{hp}-adaptive TSI}
\label{sec:hp-adaptivity}

Recall from the introduction that the parameter dependence of jumps causes singularities at every parameter $\mu$ and is therefore global, whereas the shock topology changes are typically localized in parameter space. Therefore, in this section, we introduce an $h$ and $hp$-adaptive variant of the TSI of Section \ref{sec:tsi} in order to approximate problems with changing shock topologies. See \cite{Schwab1999,Demkowicz2006} for an overview of $hp$-adaptive methods and \cite{EftangPateraRnquist2010,EftangStamm2012} for some applications to parametric functions..

It is instructive to interpret $u(x,\mu)$ as a Banach space valued function $\mu: \param \to L_1(\Omega)$ defined by $\mu \to u(\cdot, \mu)$ as in \eqref{eq:banach-space-valued}. Then, we can consider an $hp$-adaptive polynomial interpolation with respect to the parameter variable of an $L_1(\Omega)$ valued function on a finite element mesh. It causes no further difficulty to replace the cell-wise polynomial interpolation by a TSI. Indeed, let us assume that we have a partition $\grid$ of the parameter domain $\param$, such that
\begin{align*}
\param & = \bigcup_{T \in \grid} T, & \lambda(T \cap S) & = 0 \text{ for } T \ne S \in \grid,
\end{align*}
where $\lambda(T)$ denotes the Lebesgue measure of $T$. Then, on each cell $T \in \grid$, we can approximate $u$ by the simple interpolation $\interpol_{n_T} u$ defined in \eqref{eq:interpol} with degree $n_T$. For example, in Figure \ref{fig:shock-location-change}, the parameter domain is subdivided into intervals $T = (\mu_i, \mu_{i+1})$, indicated by the dotted lines so that in each horizontal strip between to dotted lines, we perform a polynomial interpolation in $\mu$ direction. Since the jump locations move in each of these strips, this method yields poor results. However, instead of a simple polynomial interpolation, we can also use the TSI $\tsi_{n_T} u$ of degree $n_T$ on each cell $T$, that is each strip in the figure. For all but the strip $(\mu_1, \mu_2)$ the number of jumps does not change, so that TSI performs well. We can then subdivide the remaining cell $(\mu_1, \mu_2)$ in order to increase the accuracy where TSI is insufficient. Thus, as in the finite element case, we have the following two possible refinements to increase the accuracy.

\begin{description}
	\item[$h$-refinement] Subdivide an cell into finitely many sub-cells.
	\item[$p$-refinement] On a cell increase the degree $n_T$ of the TSI.
\end{description}

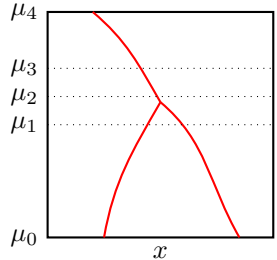
\begin{figure}[htb]

    \centering
    \begin{tikzpicture}[scale=1.5]

    \draw[thick] (0,0) rectangle (2,2);

    \node[below] at (1,0) {$x$};

    \node[left] at (0,0) {$\mu_0$};
    \node[left] at (0,1) {$\mu_1$};
    \node[left] at (0,1.250) {$\mu_2$};
    \node[left] at (0,1.5) {$\mu_3$};
    \node[left] at (0,2) {$\mu_4$};

    \draw[thick, red] (0.5,0) to[out=80, in=-120] (1, 1.2);
    \draw[thick, red] (1.7,0) to[out=120, in=-40] (1, 1.2);
    \draw[thick, red] (1,1.2) to[out=120, in=-40] (0.4, 2);

    \draw[dotted] (0,1) -- (2,1);
    \draw[dotted] (0,1.5) -- (2,1.5);
    \draw[dotted] (0,1.25) -- (2,1.25);

    \end{tikzpicture}

    \caption{Locations of the $x$ and $\mu$ components of the jump discontinuity of a not further specified piecewise smooth function $u(x,\mu)$.}
    \label{fig:shock-location-change}
\end{figure}

As in the finite element case, we need a $h$ or $hp$-refinement strategy. To this end note that on each cell T with TSI degree $n_T$, for the optimization of the inner transforms, in $\eqref{eq:opt}$ we calculate the training errors, denoted by $\epsilon^{n_T}_T$ in the following. In case of $h$ refinement this allows us to use a standard refinement criteria, such as iteratively refining the cell with largest error. Also, many $hp$-refinement strategies solemnly rely on finite element error indicators and their replacement with these training errors $\epsilon_T^{n_T}$ are appropriate candidates for a refinement strategy, see \cite{BabuskaGuo1992,Schwab1999,Demkowicz2006,MitchellMcClain2014} for an overview.

For example, the $hp$-refinement strategy proposed in \cite{SueliHoustonSchwab2000,MitchellMcClain2014}, can be adapted to our problem as follows: First, we search for the cell $T \in \grid$ for which the error indicator $\epsilon_T^{n_T}$ is largest, where $n_T$ is the current TSI degree on $T$. Then, calculate $m_T$ such that
\[
\frac{\epsilon_T^{n_T-1}}{\epsilon_T^{n_T-2}} \approx \left( \frac{n_T-1}{n_T-2} \right)^{-(m_T-1)}
\]
We $p$-refine if $n_T \le m_T-1$ and $h$-refine else.

For later reference, let us denote the outcome of the proposed $hp$-adaptive TSI variant of this section with a total number of $n$ snapshots by
\begin{equation}
\tsihp_n u.
\end{equation}

\subsection{Convergence Rates}
\label{sec:error-hp-1d}

In this section, we consider a basic error estimate of the $hp$-adaptive TSI in one parameter dimension. Higher parameter dimensions will be considered in Section \ref{sec:multiple-param-dim} below. The main goal of the error analysis is to show that in principle it is possible to achieve exponential convergence rates despite of jump discontinuities and changing topologies of the jump sets. The arguments follow the standard proofs \cite{Schwab1999}, however for interpolation in $L_\infty$ instead of the more common approximation in Hilbert spaces.

In order to keep the exposition simple, we assume a graded partition $\grid$ of the one dimensional parameter domain $\param$ into dyadic cells, so that there is a uniformly bounded number of cells $T \in \grid$ on each level, all of which are intervals of length $2^{-\ell}$ and neighboring cells differ at most by one level. In addition, all cells have a level higher that some $\ell_0$ specified below and there is one single shock topology change at $\bmu$, contained in the most refined cell of level $L$. Since we expect the TSI to be unsuccessful in the presence of the shock topology change at $\bmu$, we use use a piecewise constant approximation 
\[
  u(x,\mu) \approx u(x,\mu_T)
\]
for some $\mu_T \in T$ for all cells $T$ of level $\ell \ge L-1$, where we have included the second most refined level as a safety margin. On all other cells, we use Chebyshev nodes for the interpolation.

Similar to standard $hp$-adaptive theory \cite{Schwab1999}, we assume that for all $\mu \in \param$
\begin{equation}
  \int_\Omega \sup_{\eta \in \param} |(\bmu - \eta)^\beta \partial_\eta^p v_\mu(x,\eta) | \, dx\le p! \, C^p,
  \label{eq:hp-derivative-bounds}
\end{equation}
for some $0 \le \beta \le 1$ and for all $p \ge 1$. Note that by construction, unlike $\eta \to u(x,\eta)$, the transformed snapshots $\eta \to v_\mu(x, \eta)$, defined in \eqref{eq:transformed-snapshots}, have no jumps in $\eta$, except possibly at $\bmu$, which is taken care of by the factor $(\bmu - \eta)^\beta$.

\begin{proposition}
\label{prop:error-hp-1d}

Assume the derivatives of the transformed snapshots are bounded by \eqref{eq:hp-derivative-bounds} and that there is a Lipschitz constant $C_L$ such that 
\begin{align*}
  \|u(\cdot, \mu) - u(\cdot, \eta) \|_{L_1(\Omega)} & \le C_L |\mu - \eta|, & \mu, \eta \in \param.
\end{align*}
Let the minimal level $\ell_0$ and the degree $n_T$ on cell $T$ on level $\ell$ satisfy
\begin{align*}
  \alpha (\ell_0 +1 ) & \ge \log C, &
  n_T & = \left\lceil \frac{bL}{\ell+1} \right\rceil
\end{align*}
for some $0 \le \alpha < 1$ and $b > \beta$. Then, there are constants $c_0, c_1 \ge 0$, independent of $n$, $\beta$, $C_L$ and $C$ such that
\[
  \sup_{\mu \in \param} \|u(\cdot, \mu) - (\tsihp_n u) (\cdot,\mu)\|_{L_1(\Omega)} \le c_0 2^{-c_1 \sqrt{n}},
\]
where $n$ is the total number of snapshots.

\end{proposition}

Note that the Lipschitz continuity is only required in the $L_1$-norm and not point-wise which would be critical in the presence of jumps. In the following, we use the notation $I_\ell$ as opposed to $T$ if we want to stress that the cell is an interval on level $\ell$ and $p$ for the degree on a cell. Before we proof the proposition, we need some lemmas. The first one is used to establish a lower bound for the term $(\bmu - \eta)^\beta$ in the smoothness assumption \eqref{eq:hp-derivative-bounds}.

\begin{lemma}
  \label{lemma:level-dist}
  Let $I_\ell$ be a cell on level $\ell < L$. Then, we have
  \[
    \operatorname{dist}(I_\ell, \bmu) \ge 2^{-\ell} \left( 1- 2^{-(L-\ell) + 1} \right).
  \]
\end{lemma} 

\begin{proof}

For $\ell = L-1$, we use the trivial estimate $\operatorname{dist}(I_{L-1}, \bmu) \ge 0$. Let us consider $\ell < L-1$. By the grading property, for any cell $I_\ell$ on level $\ell$ there must be a cell $I_{\ell+1}$ of level $\ell+1$ between $I_\ell$ and $\bmu$. Since this cell has length $2^{\ell+1}$, we obtain
\[
  \operatorname{dist}(I_\ell, \bmu) \ge 2^{-(\ell+1)} + \operatorname{dist}(I_{\ell+1}, \bmu).
\]
By induction this yields
\[
  \operatorname{dist}(I_{\ell}, \bmu) \ge \sum_{k=\ell+1}^{L-1} 2^{-k}.
\]
Calculating the geometric sum proves the lemma. 
  
\end{proof}

\begin{lemma}
\label{lemma:cell-interpolation-error}

  Assume that the bounds \eqref{eq:hp-derivative-bounds} of the derivatives are satisfied. Then, on each cell $I_\ell$ of level $\ell \le L-2$, we have the error bound
  \[
  \|u(\cdot, \mu) - (\tsi_p u) (\cdot,\mu)\|_{L_1(\Omega)} \le 2^{1 + (\ell+1) \beta - (\ell+1)p} C^{p+1}
  \]
  for all $\mu \in I_\ell$.
\end{lemma}

\begin{proof}

First note, that on a cell $I_\ell$ on level $\ell$ with $L-\ell \ge 2$ we have $|I_\ell| = 2^{-\ell}$ and by Lemma \ref{lemma:level-dist} we have $|\eta - \bmu| > \frac{1}{2} 2^{-\ell}$. With the regularity assumption \eqref{eq:hp-derivative-bounds}, this implies that for every $\mu, \eta \in I_\ell$, we have
\begin{equation*}
  \int_\Omega \sup_{\eta \in I_\ell} |\partial_\eta^p v_\mu(x,\eta) | \, dx \le 2^{(\ell+1) \beta} \int_\Omega \sup_{\eta \in I_\ell} |(\bmu - \eta)^\beta \partial_\eta^p v_\mu(x,\eta) | \, dx \le 2^{(\ell+1) \beta} p! \, C^p.
\end{equation*}
From standard estimates for Chebyshev interpolation, we have
\[
v_\mu(x, \mu) - (\interpol_p v_\mu)(x,\eta) = \frac{1}{2^p (p+1)!} \left(\frac{|I_\ell|}{2} \right)^{p+1} \partial_\eta^{p+1} v_\mu(x,\xi)
\]
for some $\xi \in I_l$. This implies that
\begin{align*}
  \|u(\cdot, \mu) - (\tsi_p u) (\cdot,\mu)\|_{L_1(\Omega)} & = \int_\Omega | v_\mu(x, \mu) - (\interpol_p v_\mu)(x,\eta)| \, dx \\
  & \le \frac{1}{2^p (p+1)!}\left(\frac{|I_\ell|}{2} \right)^{p+1} \int_\Omega \sup_{\eta \in I_\ell} |\partial_\eta^{p+1} v_\mu(x,\eta) | \, dx \\
  & \le 2^{-p + (\ell+1) \beta - (\ell+1)(p+1)} C^{p+1}.
\end{align*}

\end{proof}

Now we are ready to prove Proposition \ref{prop:error-hp-1d}.

\begin{proof}[Proof of Proposition \ref{prop:error-hp-1d}]

Let
\[
  \sigma_\ell := \sup_{\mu \in I_\ell} \|u(\cdot, \mu) - (\tsi_p u) (\cdot,\mu)\|_{L_1(I_\ell)}
\]
be the maximal error on a cell $I_\ell$ on level $\ell$. We first consider the error for $\ell_0 \le \ell \le L-2$. By Lemma \ref{lemma:cell-interpolation-error}, we have
\[
  \sigma_l \le 2^{1 + (\ell+1) \beta - (\ell+1)p + (\log C)(p+1)}.
\]
By our choice of the degrees $p$, we have 
\[
\frac{bL}{\ell+1} \le p \le 1+ \frac{bL}{\ell+1},
\]
which yields
\[
  \sigma_l \le 2^{1 + (\ell+1) \beta - bL + (\log C)\left( 2 + \frac{bL}{\ell+1} \right)}.
\]
The function $\real \ni \ell \to 1 + (\ell+1) \beta + (\log C)\left( 2+ \frac{bL}{\ell+1} \right)$ has one local minimum and attains its maximum on the boundary. For $\ell = \ell_0$, it is $1 + (\ell_0+1) \beta + 2\log C + \frac{\log C}{\ell_0+1} bL \le 1+(\ell_0+1) \beta + 2 \log C + \alpha b L$ and for $\ell = L-2$, it is bounded by $1+(L-1)\beta + 2(1+b) \log C$, so that we obtain
\[
  \sigma_l \le \max\{C^2 2^{1+(\ell_0+1) \beta} , 2^{1-\beta} C^{2(1+b)}\} \max\{ 2^{-(1-\alpha) b L}, 2^{-(b-\beta) L}\}.
\]
For $\ell \ge L-2$, from the Lipschitz continuity of $u(\cdot, \mu)$ in the $L_1(\Omega)$ norm, we conclude that
\[
  \sigma_\ell \le C_L |\mu - \mu_{I_\ell}| \le C_L 2^{-\ell} \le 2 C_L 2^{-L}
\]
In conclusion, on every cell, the error is bounded by $c_0 2^{-c_1 L}$ for suitable constants $c_0$ and $c_1$. Finally, from the assumptions, one easily verifies that $n \le \gamma^2 L \log L \le \gamma^2 L^2$ for some constant $\gamma \ge 0$ and thus $L \gtrsim \gamma \sqrt{n}$, which concludes the proof.
\end{proof}

\subsection{Singular Transforms?}
\label{sec:singular}

Even if we can align the discontinuities away from the parameter locations of the shock topology changes, the necessary transforms typically have singularities themselves. This can be easily seen in Figure \ref{fig:transform-singularity} for two colliding shocks. Say we want to evaluate the TSI at parameter $\bmu$ of the shock topology change in the figure. To this end, we use evaluations of the snapshots $u(\cdot, \eta)$, $\eta \in \param_n$ at the positions $\phi(\bmu, \eta)(x)$. Two such transforms are shown in the figure. Clearly, if we start a little bit to the left of the shocks or a little bit to the right of the shocks the transforms $\phi(\bmu, \eta)(x)$ must have vastly different values, i.e. the transform $x \to \phi(\bmu, \eta)(x)$ has a jump.

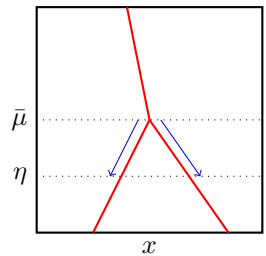
\begin{figure}[htb]

    \centering
    \begin{tikzpicture}[scale=1.5]

    \draw[thick] (0,0) rectangle (2,2);

    \node[below] at (1,0) {$x$};
    
    \draw[dotted] (0,1) node[left] {$\bmu$} -- (2,1);
    \draw[dotted] (0,0.5) node[left] {$\eta$} -- (2,0.5);
  
    \draw[thick, red] (0.5,0) -- (1, 1);
    \draw[thick, red] (1.7,0) -- (1, 1);
    \draw[thick, red] (1,1) -- (0.8, 2);

    \newcommand{\figH}{0.1}
    \draw[->, blue] (1-\figH, 1) -- (0.75 -\figH, 0.5);
    \draw[->, blue] (1+\figH, 1) -- (1.35+\figH, 0.5);

    \end{tikzpicture}

    \caption{Red lines: location of the jumps in the $x-\mu$-plane. Blue arrow: transform $\phi(\mu, \eta)(x)$. The direction at hight $\mu$ is given by $\Phi(\mu, x)$.}
    \label{fig:transform-singularity}
\end{figure}

Nonetheless, in this section we argue that this does not necessarily induce high gradients or jumps in the $hp$-adaptive TSI's transforms. This is important for two reasons: First, we have to find practical approximations of $x \to \phi(\mu, \eta)(x)$, which becomes complicated if this function has jumps itself. Second, due to this approximation, we also have to consider the stability of the $hp$-adaptive TSI with respect to perturbations in the transforms, which also deteriorates with their regularity.

We still consider one parameter dimension and defer higher dimensions to Section \ref{sec:multiple-param-dim}, below. It will be useful to interpret the transforms as the solution of an ODE as established by the following lemma.

\begin{lemma}
\label{lemma:ode}
  Assume that the transform $\phi(\mu, \eta)$ satisfies the decomposition \eqref{eq:transform-decomposition}. Then, it is the solution of the ODE
  \begin{align}
    \frac{d}{d \eta} \phi(\mu, \eta)(x) & = \Phi(\eta, x) & \phi(\mu, \mu)(x) & = x
    \label{eq:transform-ode}
  \end{align}
  with $\Phi(\eta, x) := \frac{\partial}{\partial \eta} \phi(\mu, \eta)(x)$.
\end{lemma}

\begin{proof}

By \eqref{eq:transform-decomposition} we have
\[
  \phi(\mu, \eta) = \phi(\xi, \eta) \circ \phi(\mu, \xi)
\]
for all $\xi \in \param$. With $y := \phi(\mu, \eta)(x)$ this implies that
\[
  \phi(\mu, \eta+h)(x) = \phi(\eta, \eta+h) \circ \phi(\mu, \eta)(x) = \phi(\eta, \eta+h)(y).
\]
Since by \eqref{eq:interpolation-condition} we have $\phi(\mu, \eta)(x) = y = \phi(\eta, \eta)(y)$ this implies that
\begin{multline*}
  \lim_{h \to 0} \frac{1}{h} \Big[\phi(\mu, \eta+h)(x) - \phi(\mu, \eta)(x) \Big] \\
   = \lim_{h \to 0} \frac{1}{h} \Big[\phi(\eta, \eta+h)(y) - \phi(\eta, \eta)(y) \Big] = \partial_\eta \phi(\eta, \eta)(y).
\end{multline*}

\end{proof}

If the function $\Phi(\eta, x)$ is Lipschitz continuous in the ``initial values'' $x$, the Gronwall lemma ensures Lipschitz continuity of $x \to \phi(\mu, \eta)(x)$, but for changing shock topologies, this assumption is too strong. E.g in Figure \ref{fig:transform-singularity} we have already seen that $x \to \phi(\mu, \eta)(x)$ has a jump. Likewise, $\Phi$ is the corresponding direction field and has a jump at the same location.

However, we also see that this loss of Lipschitz continuity is only necessary at the one parameter where the shock topology changes, denoted by $\muSingular$. We now assume that approaching this parameter the Lipschitz continuity does not blow off too quickly by requiring that $|\eta - \muSingular| \Phi(\eta, x)$ is Lipschitz continuous uniformly in $\eta$. Then, we obtain the following modified stability result.

\begin{lemma}
  \label{lemma:stability}

  Assume that $|\eta - \muSingular| \Phi(\eta, x)$ is Lipschitz continuous with respect to $\eta$ with Lipschitz constant $L$ and that $\muSingular \not\in [\mu, \eta]$. Then, we have
  \[
    |\phi(\mu, \eta)(x) - \phi(\mu, \eta)(y)| \le \left| \frac{\eta-\muSingular}{\mu-\muSingular} \right|^L |x-y|.
  \]

\end{lemma}

\begin{proof}

We compute the difference of two ODE solutions with different initial values along standard lines: Rewriting the ODEs \eqref{eq:transform-ode} for $x_\eta := \phi(\mu, \eta)(x)$ and $y_\eta := \phi(\mu, \eta)(y)$ as an integral equation, subtracting and taking absolute values, we obtain
\begin{align*}
  |x_\eta - y_\eta| & \le |x-y| + \int_\mu^\eta |\Phi(\xi, x_\xi) - \Phi(\xi, y_\xi)| \, d \xi \\
  & \le |x-y| + \int_\mu^\eta L \frac{|x_\xi - y_\xi|}{|\xi-\muSingular|} \, d \xi
\end{align*}
By the Gronwall inequality, this implies that
\[
  |x_\eta - y_\eta| \le |x-y| e^{\int_\mu^\eta \frac{L}{|\xi-\muSingular|} \, d \xi} = |x-y| \left| \frac{\eta-\muSingular}{\mu-\muSingular} \right|^L
\]
which completes the proof.

\end{proof}

We see that the transforms $x \to \phi(\mu, \eta)(x)$ are Lipschitz continuous with Lipschitz constant that depends on the ratio $\frac{\eta-\muSingular}{\mu-\muSingular}$ of the distances of $\mu$ and $\eta$ to the singularity. The important observation in our context is that for an $hp$ adaptive scheme this ratio is bounded for all cells on which we compute a TSI. Indeed, with the assumptions of Section \ref{sec:error-hp-1d}, we use a piecewise constant approximation on cells with level $L$ and $L-1$ and on all other cells $T$ on level $\ell < L-1$ by Lemma \ref{lemma:level-dist}, we have $|\mu - \bmu| \ge \operatorname{dist}(\bmu, T) \ge \frac{1}{2} 2^{-\ell}$ so that
\[
  \frac{|\eta-\muSingular|}{|\mu-\muSingular|} \le \frac{|\eta-\mu|}{|\mu-\muSingular|} + \frac{|\mu-\muSingular|}{|\mu-\muSingular|} \le \frac{2^{-\ell}}{\frac{1}{2} 2^{-\ell}} + 1 = 3,
\]
so that the Lipschitz constant is indeed bounded on all relevant cells.

Although, in this paper we assume that there is a transform $\phi$ that perfectly aligns the jumps of $u$ away form the shock topology changes, in practice we can only compute a perturbation $\varphi$ thereof, with errors coming from the limited number of degrees of freedom in its approximation and the numerical optimization of \eqref{eq:opt}. The paper \cite[Section 3]{Welper2015} contains an analysis of this perturbation error of a single cell TSI \eqref{eq:tsi}, stating that
\[
  \|u_n(\cdot, \mu; \phi) - u_n(\cdot, \mu; \varphi)\|_{L_1(\Omega)} \le c \sup_{\eta \in \param} \|\phi(\mu, \eta) - \varphi(\mu, \eta)\|_{L_1(\Omega)},
\]
with a constant $c$ that depends on the TSI degree (Lebesgue constant), the total variation norm of the snapshots $u(\cdot, \eta)$, $\eta \in \param_n$ and some properties of the transforms. The most important assumption on the latter is that
\[
  \phi(\mu, \eta)_* \lambda (A) \le \gamma \lambda(A)
\]
for every measurable set $A \subset \Omega$, some constant $\gamma \ge 0$, Lebesgue measure $\lambda$ and push-forward measure $\phi(\mu, \eta)_* \lambda$. In the $h$ or $hp$-adaptive case, this stability estimate can be applied cell-wise and the latter assumption can be directly verified by Lemma \ref{lemma:stability}. Indeed, by the decomposition \eqref{eq:transform-decomposition} we have that $\phi(\mu, \eta)^{-1}(x) = \phi(\eta, \mu)(x)$ so that we obtain
\[
  \phi(\mu, \eta)_* \lambda (A) = \lambda\big( \phi(\eta, \mu)(A) \big) \le \left| \frac{\eta-\muSingular}{\mu-\muSingular} \right|^L \lambda(A),
\]
where the latter estimate uses the Lipschitz constant form Lemma \ref{lemma:stability} and well known estimates of Lebesgue measures under Lipschitz maps, see e.g. \cite{EvansGariepy2015}. As argued above, in reasonable scenarios this Lipschitz constant is uniformly bounded and we obtain a stable method.

\section{Higher Parameter Dimensions}
\label{sec:multiple-param-dim}

In this section, we extend the $h$ or $hp$ adaptive TSI of the last section to higher parameter dimensions. 

\subsection{Coordinate wise TSI}
\label{sec:tensor-tsi}

\paragraph{The method} As already discussed in the introduction, the local refinement strategy of Section \ref{sec:hp-adaptivity} is questionable in higher parameter dimensions. Therefore, in this section we discuss a tensor product type construction and its merits for high parameter dimensions. To this end, let us assume that the parameter domain is a Cartesian product $\param = \param^1 \times \cdots \times \param^d$. Then, we can easily construct an interpolation operator for Banach space valued functions by
\begin{equation}
\interpol_{n_1} \otimes \cdots \otimes \interpol_{n_d},
\label{eq:tensor-interpol}
\end{equation}
where each $\interpol_{n_i}$ is a polynomial interpolation operator \eqref{eq:interpol} of degree $n_i$ acting on $\param^i$. In principle, we can apply a similar construction to the TSI interpolation operator \eqref{eq:tsi-operator}. However, a tensor product is not directly applicable because $\tsi_n$ is not linear. To this end, note that we can rewrite the tensor product interpolation \eqref{eq:tensor-interpol} as iteratively applying an interpolation to each parameter coordinate:
\begin{equation*}
\interpol_{n_1} \otimes \cdots \otimes \interpol_{n_d} = (\id \otimes \cdots \otimes \id \otimes \interpol_{n_d}) \cdots  (\interpol_{n_1} \otimes \id \otimes \cdots \otimes \id),
\end{equation*}
where $\id$ is the identity operator. Although, we cannot build a tensor product of the non-linear TSI interpolation operator \eqref{eq:tsi-operator}, we can still apply the TSI in the latter component-wise sense. To this end, we first compute snapshots $u(\cdot, \boldsymbol{\eta})$ for the interpolation points $\boldsymbol{\eta} \in \param_{\boldsymbol{n}} := \param_{n_1}^d \times \dots \times \param_{n_d}^d$. Applying the TSI to the first component yields
\begin{equation*}
u(x,\mu_1, \eta_2, \dots, \eta_d) \approx (\tsi_{n_1}^1 u)(x, \mu) := \sum_{\eta_1 \in \param_{n_1}^1} \ell_{\eta_1}(\mu_1) u(\phi(\mu_1, \eta_1)(x), \eta_1, \eta_2, \dots, \eta_d).
\end{equation*}
Afterwards, we know approximations of $u(x, \mu_1, \eta_2, \dots, \eta_d)$ for all $x \in \Omega$, $\mu_1 \in \param^1$ but only for discrete $\eta_i \in \param_{n_i}^i$, $i=2, \dots, d$.

We can now apply the TSI with respect to the second variable $\mu_2$. We can do so by using transforms $\phi(\mu_2, \eta_2)(x)$. However, in the $h$-adaptive case it will be crucial that we allow some more flexibility and actually use transforms $\phi(\mu_2, \eta_2)(x,\mu_1) \in \Omega \times \param^1$ that also transform the first parameter variable. This requires us to evaluate the ``snapshots'' $(\tsi_{n_1}^1 u)(\tilde{x}, \tilde{\mu}_1, \eta_2, \dots, \eta_d)$ at any transformed values $(\tilde{x}, \tilde{\mu}_1) = \phi(\mu_2, \eta_2)(x, \mu_1)$. This does not cause any difficulties because we know those values form our fist TSI in $\mu_1$ direction.

Continuing this algorithm, let us denote by
\[
(\tsi_{n_i}^{i, \boldsymbol{j}} u)(x, \boldsymbol{\mu}) = (\tsi_{n_i}^{i, j_1 \dots j_r})(x, \boldsymbol{\mu})
\]
the result of the TSI with respect to the variable $\mu_i$ where the variables $x$ and $\boldsymbol{\mu}_{\boldsymbol{j}} := (\mu_{j_1}, \dots, \mu_{j_r})$ are transformed. Then, we obtain a \emph{component-wise TSI}
\begin{equation}
\tsi_{\boldsymbol{n}} := \tsi_{n_d}^{d, 1 \cdots d-1} \circ \cdots \circ \tsi_{n_1}^{1,}
\label{eq:tensor-tsi}
\end{equation}
with degrees $\boldsymbol{n} = (n_1, \dots, n_d)$. This component-wise TSI depends on the ordering of involved TSIs, which is left for future research.

\paragraph{Computing the transforms}

Each TSI $\tsi_{n_1}^{1, 1 \dots i-1}$ is applied to a function $w(x, \mu_1, \dots, \mu_{i-1}, \eta_i, \dots, \eta_d)$ that is the output of previous steps of the algorithm, where $x$, $\mu_1, \dots, \mu_{i-1}$ are transformed and $\eta_{i+1}, \dots, \eta_d$ are frozen discrete interpolation points. To find the necessary transforms, as in \eqref{eq:opt}, we solve the optimization problem
\begin{equation}
\phi = \argmin_{\phi \in \transforms} \, \sup_{\eta_i \in \param_T^i} \|(Id - \tsi_{n_i}^{i, \boldsymbol{j}})w(\cdot, \eta_i, \eta_{i+1}, \dots, \eta_d)\|_{L_1(\Omega \times \param^{\boldsymbol{j}})},
\label{eq:opt-tensor}
\end{equation}
where $Id$ is the identity, $\transforms$ is a set of admissible transforms and $\param_T^i \subset \param^i$ is a set of training parameters. In order to provide the necessary input for the TSIs of the remaining parameter components, we have to solve all independent optimization problems for all interpolation points $\eta_{i+1}, \dots, \eta_d$.

First note that in the coordinates $\param^1 \times \dots \times \param^{i-1}$ we use the $L_1$-norm as e.g. opposed to the $L_\infty$-norm that we usually use for the parameter. As shown in \cite{Welper2015}, this ensures that the objective function is Lipschitz continuous. In particular, via the Rademacher Theorem this implies that it is differentiable almost everywhere, which allows us to use gradient based optimizers.

To calculate the necessary $L_1$-norms, we must be able to evaluate $w$ for all training parameters $\eta_i \in \param^i_T$. However, form the TSIs in the previous coordinates we only know $w$ for the interpolation points $\param^i_{n_i}$ which must be distinct from the training samples. We can resolve this issue in two ways.

First, we can calculate additional TSIs 
\[
  \tsi_{n_k}^{i-1, 1 \cdots i-2} \circ \cdots \circ \tsi_{n_1}^{1} u
\]
not only for $\mu_i$ frozen to values in $\param^i_{n_i}$ but also for $\mu_i$ in $\param^i_T$. This ensures that we can evaluate the integrand in the $L_1$-norm in \eqref{eq:opt-tensor} at all points of the integration domain, so that we can use any quadrature rule.

The second solution is to note that $w(x, \mu_1, \dots, \mu_{i-1}, \mu_i, \eta_{i+1}, \dots, \eta_d)$ is an TSI approximation of $u(x, \mu_1, \dots, \mu_{i-1}, \mu_i, \eta_{i+1}, \dots, \eta_d)$. We can resort back to the parametric PDE and compute this function at all required quadrature points. Of course, this severely limits their number in the $\mu_1, \dots, \mu_{i-1}$ variables. However, the goal of the optimization is to obtain a suitable transform, not the accuracy of the $L_1$-error. This can often be achieved by a very coarse quadrature as shown in the numerical experiments below.

\paragraph{Number of degrees of Freedom} Let us examine the number of degrees of freedom involved in this method. For the online phase, we have to store the snapshots and all transforms involved in the component-wise TSI. For simplicity, we assume that for each coordinate direction we use the same number $n = n_i$ of interpolation points. Then, we need $n^k$ snapshots on a tensor grid. In order to account for the transforms, we need a discrete representation thereof. More precisely, for each interpolation point $\eta_i \in \param^i_n$ we have to approximate the transform $[(x, \mu_1, \dots, \mu_{i-1}), \mu_i] \to \phi(\mu_i, \eta_i)(x, \mu_1, \dots, \mu_{i-1})$. A simple possibility is to choose a basis $\psi_j(x)$, $j=1,\dots,K$ for the physical axis and simply reuse the interpolation with nodes $\param^i_n$ for each parameter axis, which yields
\begin{equation}
  \phi(\mu_i, \eta_i)(x, \mu_1, \dots, \mu_{i-1}) \approx \sum_{\boldsymbol{\eta} \in \param^i_n\times \cdots \times \param^1_n} \ell_{\eta_i}(\mu_i) \cdots \ell_{\eta_1}(\mu_1) \left( \sum_{j=1}^K c_{\boldsymbol{\eta}, j} \psi_j(x)
  \label{eq:discrete-transform}
\right).
\end{equation}
Then, we need $n^i K$ degrees of freedom for each transform for coordinate axis $i$. There are $n$ such transforms, for each $\eta_i \in \param^i_n$, and $i=1, \dots d$ axes so that we need
\[
  N = \sum_{i=1}^{d} n n^i K = n \frac{n^{d+1}-1}{n-1} K -nK\le 2 n^{d+1} K
\]
degrees of freedom. For a comparison, without the coordinate-wise application of the TSI, we need $n^d K $ degrees of freedom for each of the $n^d$ transform $(x, \boldsymbol{\mu}) \to \phi(\boldsymbol{\mu}, \boldsymbol{\eta})(x)$, $\boldsymbol{\eta} \in \param_{\boldsymbol{n}}$ and thus $n^{2d}K$ in total, which is substantially more that in the component-wise case. Nonetheless, also in the latter case the number of degrees of freedom for the transforms and the number of snapshots scale exponentially with the dimension $d$, which is expected for a full tensor product type approximation.

Therefore, component-wise TSI cannot deal with the curse of dimensionality. However, it is possible to use different polynomial degrees in each coordinate direction and to consider linear combinations of various such choices. This is the starting point for sparse grids \cite{BungartzGriebel2004} or sparse polynomial expansions \cite{CohenDeVoreSchwab2010,CohenDeVoreSchwab2011,CohenDeVore2015}. How such techniques can be combined with component-wise TSI is beyond the scope of this paper and will be treated in future research.

\subsection{Coordinate wise \texorpdfstring{$hp$}{hp}-adaptive TSI}
So far, we have ignored $h$-refinement as described in Section \ref{sec:hp-adaptivity}. Before we incorporate them into our method, let us first consider the classical case again. In contrast to non-adaptive interpolation, we usually do not ``tensorize'' $1d$ $hp$-adaptive finite element methods to obtain approximations for higher dimensional functions. This only makes sense if the singularities are aligned with coordinate axes. However, for the TSI the situation is different: It is designed to align singularities with the coordinate axes! In fact that was our starting point in Figure \ref{fig:tsi-motivation}: We chose the inner transform in a way the the jump singularity is ``invisible'' in the parameter direction, which is equivalent to be aligned with the coordinate axis. Therefore, replacing the TSIs $\tsi_{n_i}^{i, 1, \dots, i-1}$ by its $hp$-adaptive version from Section \ref{sec:hp-adaptivity} in the component-wise TSI \eqref{eq:tensor-tsi} makes perfect sense.

This concept is illustrated in Figure \ref{fig:tensor-adaption}. The problem is the same as in Figure \ref{fig:shock-location-change} with one added parameter dimension so that we deal with a function $u(x, \mu_1, \mu_2)$. For increasing $\mu_1$, we start out with two shocks which then collide into one shock. If we change the second parameter $\mu_2$, we have the same behavior with different jump and collision locations. The figure shows the jump locations in the $x,\mu_1$ plane for two different values of $\mu_2$ (solid and dashed lines). For this problem, we can reconstruct $(x, \mu_1) \to u(x, \mu_1, \mu_2)$ for fixed $\mu_2$ by the $hp$-adaptive TSI just as in Section \ref{sec:hp-adaptivity}. For the second TSI in the $\mu_2$ direction, we apply a transform for the entire $x,\mu_1$-plane. That allows us to align the dashed jump set exactly with the solid jump set in the figure. After alignment the jumps are again ``invisible'' in the $\mu_2$-direction and interpolation is expected to yield good results. Note that in case the jump-sets in the $x,\mu_1$-plane cannot be matched, we can localize by an $hp$-adaptive strategy in the $\mu_2$-variable.
\begin{figure}[htb]

  \hfill
  \begin{tikzpicture}[scale=1.5]

	\draw[thick] (0,0) rectangle (2,2);

	\node[below] at (1,0) {$x$};
	\node[below] at (1,0) {$\phantom{\mu^2}$};
	\node[left] at (0,1) {$\mu_1$};

	\draw[thick, red] (0.5,0) to[out=80, in=-120] (1, 1.2);
	\draw[thick, red] (1.7,0) to[out=120, in=-40] (1, 1.2);
	\draw[thick, red] (1,1.2) to[out=120, in=-40] (0.4, 2);

	\draw[thick, dashed, red] (0.5,0) to[out=80, in=-120] (0.8, 1.3);
	\draw[thick, dashed, red] (1.7,0) to[out=120, in=-40] (0.8, 1.3);
	\draw[thick, dashed, red] (0.8,1.3) to[out=120, in=-40] (0.2, 2);

  \end{tikzpicture}
  \hfill
  \begin{tikzpicture}[scale=1.5]

        \fill[thick, color=black!20] (0,0) rectangle (2,2);        
        \node[below] at (1,0) {$\mu^2$};
        \node[left] at (0, 1) {$\mu^1$};
        
        \fill[color=black!80] (0,0) -- (0,1) to[out=30, in=210] (2,1) -- (2,0) -- (0,0);
        
  \end{tikzpicture}
  \hfill~

	\caption{Left: Jump locations of a function $u(x, \mu_1, \mu_2)$ in the $x, \mu_1$-plane for two different values of $\mu_2$ (dashed and solid lines). Right: Regions in parameter space with two and one jump.}
	\label{fig:tensor-adaption}
\end{figure}
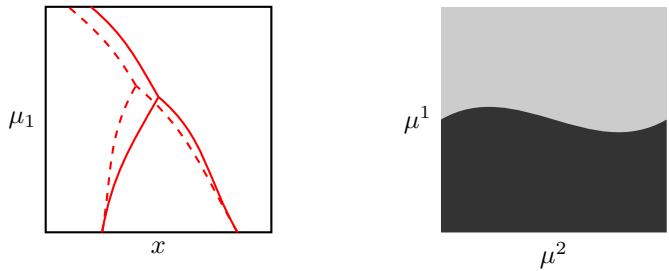

Thus, similar to the non-adaptive case, let us denote by
\[
(\tsihp_{n_i}^{i, j_1 \dots j_r})(x, \mu_1, \dots, \mu_d)
\]
the result of the $hp$-adaptive TSI with respect to the variable $\mu_i$, where the variables $x$ and $\mu_i$, $i \in \{j_1, \dots, j_r\}$ are transformed and the remaining variables are frozen. Then, we obtain the \emph{coordinate-wise $hp$-adaptive TSI} by the composition
\begin{equation}
\tsi_{\boldsymbol{n}} := \tsihp_{n_d}^{d, 1 \cdots d-1} \circ \cdots \circ \tsihp_{n_1}^{1,}
\label{eq:tensor-hptsi}
\end{equation}
with target accuracies $\boldsymbol{n} = (n_1, \dots, n_d)$.

\subsection{Comparison of \texorpdfstring{$hp$}{hp}-TSI and component-wise \texorpdfstring{$hp$}{hp}-TSI}

In principle both the $hp$-TSI and the component-wise $hp$-TSI can be applied to parameter dimensions larger than one. In this section we highlight some differences.

In $hp$-adaptive finite elements, interpolation or TSI, we choose interpolation points subject to the partition $\grid$ of the domain $\param$. In order to make this more concrete, let $\param_N$ be the set of all interpolation points involved in the $hp$-adaptive approximation. Then, for a parameter $\mu$, the active interpolation points $A(\mu) \subset \param_N$ that are used for the approximation of $u(\cdot, \mu)$ are completely determined by the partition, i.e.
\begin{equation}
  A(\mu) = \param_N \cap T
  \label{eq:partition-active-points}
\end{equation}
in case $\mu \in T \in \grid$. Although based on local refinements, such a partitioning of the active coefficients is not possible for the component-wise TSI. This can be easily be seen in Figure \ref{fig:no-partition} in case of two parameter dimensions. We first freeze the value of $\mu_2$ to two different interpolation points and perform an $hp$-adaptive interpolation in the $\mu_1$ coordinate. Let us assume that for the left value of $\mu_2$ we have one cell subdivision and none for the right value of $\mu_2$. Afterwards, we know all values of $u(x, \mu_1, \mu_0)$ along the dashed lines and perform an interpolation in the $\mu_2$ direction. For simplicity, we assume that all transforms are identities. If all interpolations are of first order, this results e.g. in the interpolation points $\mu^1, \dots, \mu^6$ shown in the figure.

For the interpolation at the point $a$, we thus obtain the active indices $A(a) = \{\mu^1, \mu^2, \mu^5, \mu^6\}$. Likewise, for the interpolation of $b$, we obtain $A(b) = \{\mu^3, \mu^4, \mu^5, \mu^6\}$. It follows that the interpolations at $a$ and $b$ share some active coefficients but not all. This is not possible for a partition based selection \eqref{eq:partition-active-points} of active indices. 

It is well understood that adaptive partitions are questionable in higher parameter dimensions. However, as the argument shows, the component-wise $hp$ adaptive TSI is not based such partitions in the usual sense. What this entails for the prospects to use the method (or rather extensions thereof as discussed at the end of Section \ref{sec:tensor-tsi}) in higher parameter dimensions is left for future research.

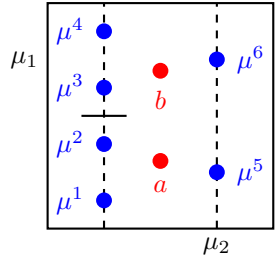
\begin{figure}[htb]

	\centering
	\begin{tikzpicture}[scale=1.5]

	\draw[thick] (0,0) rectangle (2, 2);
s
	\node[left] at (0, 1.5) {$\mu_1$};
	\node[below] at (1.5, 0) {$\mu_2$};

	\draw[thick, dashed] (0.5,0) -- (0.5,2);
	\draw[thick, dashed] (1.5,0) -- (1.5,2);

	\draw[thick] (0.3,1) -- (0.7,1);

	\fill[blue] (0.5, 0.25) circle[radius=2pt] node[left=4pt] {$\mu^1$};
	\fill[blue] (0.5, 0.75) circle[radius=2pt] node[left=4pt] {$\mu^2$};

	\fill[blue] (0.5, 1.25) circle[radius=2pt] node[left=4pt] {$\mu^3$};
	\fill[blue] (0.5, 1.75) circle[radius=2pt] node[left=4pt] {$\mu^4$};

	\fill[blue] (1.5,0.5) circle[radius=2pt] node[right=4pt] {$\mu^5$};
	\fill[blue] (1.5,1.5) circle[radius=2pt] node[right=4pt] {$\mu^6$};

	\fill[red] (1,0.6) circle[radius=2pt] node[below=4pt] {$a$};
	\fill[red] (1,1.4) circle[radius=2pt] node[below=4pt] {$b$};

	\end{tikzpicture}

	\caption{Example interpolation points of tensor $hp$ TSI}
	\label{fig:no-partition}
\end{figure}

\subsection{Error Estimates}

In this section, we establish some preliminary error analysis of the component-wise $hp$-adaptive TSI. In principle, we can apply the $hp$-adaptive TSI directly for moderate parameter dimensions, without the tensor-type construction of this section. However, this limits the convergence rate to low polynomial orders. Indeed, e.g. in the example in Figure \ref{fig:tensor-adaption}, we have two regions in parameter space, separated by a different number of shocks. Therefore, an $h$-adaptive TSI, applied to the two dimensional parameter space, is akin to the finite element approximation of functions with a curve-like singularity (in 2d). It is understood that such approximations have limited orders. The intuitive reason is that the jump-curve must be approximated by the shape of the finite element cells, which have linear boundaries, and therefore a limited order, irrespective of the polynomial degree inside of the cells. In contrast, the alignment of the singularities to coordinate axis in the component-wise $hp$-adaptive TSI can be carried out with arbitrary order. To flesh out the argument, in this section we provide an error analysis for a simple model case.

To this end, assume that the parameter domain can be split into two components $\param \times \hat{\param}$, with $\param \subset \real$ and $\hat{\param} \subset \real^{d-1}$ and two regions 
\begin{align*}
  R_1 & = \{ (\mu, \hmu): \,  \mu \le \bmu(\hmu) \}, & 
  R_2 & = \{ (\mu, \hmu): \,  \mu \ge \bmu(\hmu) \}, 
\end{align*}
for some function $\bmu: \hat{\param} \to \param$ so that in each region there is a transform with:
\begin{align}
  \phi(\mu, \hmu, \eta, \heta)(x) & \in J_{\eta, \heta} \Leftrightarrow x \in J_{\mu, \hmu}, & (\mu, \hmu), (\eta, \heta) & \in R_i, & i & = 1,2,
  \label{eq:jump-align}
\end{align}
where $J_{(\mu, \hmu)} \subset \Omega$ is the set of jump locations of $x \to u(x, \mu, \hmu)$. This condition implies that the transformed snapshots $u(\phi(\mu, \hmu, \eta, \heta)(x), \eta, \heta)$ have jumps only in the jump set $J_{(\mu, \hmu)}$ of the target parameter, i.e. informally the transform aligns the jumps. Also note that theses transforms only align jumps for parameters $(\mu, \hmu)$ and $(\eta, \heta)$ in the same region $R_i$.

We now perform a component-wise TSI, $hp$-adaptive in $\mu$ and non-adaptive in $\hmu$, which is sufficient in our model case so that
\[
  u(x, \mu, \hmu) \approx (\tsihp_n^2 \circ \tsi_n^1 u)(x, \mu, \hmu),
\]
where $\tsihp_n^1$ is the $hp$-adaptive TSI applied to the $\mu$ parameter with $n$ adaptively chosen interpolation points and $\tsi_n^2$ is the non-adaptive TSI applied to $\hmu$ and $n$ snapshots on a tensor grid. Based on \eqref{eq:jump-align}, we first construct the transforms involved in the respective TSIs and then provide a crude error analysis that shows the desired exponential decay.

\paragraph{Construction of the transforms}

For the $\tsi_n^1$, we choose the transform $\phi(\mu, \heta, \eta, \heta)$ restricted to frozen interpolation point $\heta$. For $\tsihp_n^2$, a simple restriction of \eqref{eq:jump-align} is insufficient because it can mix up the two regions with different jump set topology. Therefore, we choose functions $\varphi(\hmu, \heta): \param \to \param$ that are monotonically increasing such that
\[
  \varphi(\hmu, \heta)(\bmu(\hmu)) = \bmu(\heta).
\]
Note that for fixed $\hmu$ and $\heta$, these functions are univariate and scalar valued with one given point value and therefore easy to construct. Now, we define
\begin{equation}
  \phi(\hmu, \heta)(x,\mu) := \big[ \phi(\mu, \hmu, \varphi(\hmu, \heta)(\mu), \eta)(x), \quad \varphi(\hmu, \heta)(\mu) \big].
  \label{eq:transform2}
\end{equation}
Just as the transform in \eqref{eq:jump-align}, we denote this new transform by $\phi$, which is not ambiguous because of the different argument lists. We have to verify that this transform aligns the jumps. As can be seen e.g. in Figure \ref{fig:tensor-adaption}, the jump set for the $\hmu$-component TSI is given by
\[
  J_{\hmu} = \{ (x,\mu): \, x \in J_{\mu, \hmu}\},
\]
so that we have to ensure that
\[
  \phi(\hmu, \heta)(x) = J_{\heta} \quad \Leftrightarrow \quad x \in J_{\hmu}.
\]
To this end, note that for $(x,\mu) \in J_{\hmu}$ and $\eta = \varphi(\hmu, \heta)(\mu)$, by the definition of $\varphi$, we have $(\mu, \hmu) \in R_i$ if and only if $(\eta, \heta) \in R_i$, $i=0,1$ and thus  by \eqref{eq:jump-align} we obtain
\[
  \phi(\hmu, \heta)(x,\mu) \in J_{\eta, \heta} \times \{\eta\} \subset J_{\heta}.
\]
The case $(x,\mu) \not\in J_{\hmu}$ works analogously.

\paragraph{Error Analysis}

Now that we have established the existence of all necessary transforms, in view of Proposition \ref{prop:error-hp-1d} and Chebyshev interpolation on tensor grids \cite{GassGlauMahlstedtMair2015}, we assume that all component TSIs converge exponentially, uniformly in the respective frozen parameters:
\begin{equation}
  \begin{aligned} 
    \sup_{\mu \in \param} \|u(\cdot, \mu, \hmu) - (\tsihp_n^1 u)(\cdot, \mu, \hmu) \|_{L_1(\Omega)} & \le C \rho^{-c n^{1/2}}, & \hmu \in \hat{\param}    \\
     \sup_{\hmu \in \hat{\param}} \|u(\cdot, \mu, \hmu) - (\tsi_n^2 u)(\cdot, \mu, \hmu) \|_{L_1(\Omega)} & \le C \rho^{-n^{1/(d-1)}}, & \mu \in \param,
  \end{aligned}
  \label{eq:rate-assumption}
\end{equation}
for some constants $c,C \ge 0$.  Let
\[
  \Lambda_n := \sup_{\hmu \in \hat{\param}} \sum_{\heta \in \hat{\param}_n} |\ell_\heta(\hmu)|.
\]
be the Lebesgue constant of the interpolation $\tsi^2$ and assume that
\begin{equation}
  \phi(\mu, \hmu, \varphi(\hmu, \heta)(\mu), \eta)_* \lambda(A) \le \gamma \lambda(A),
  \label{eq:stab}
\end{equation}
for some constant $\gamma \ge 0$, where $\lambda$ is the Lebesgue measure and $\phi(\mu, \hmu, \varphi(\hmu, \heta)(\mu), \eta)_* \lambda$ the push-forward measure induced by the transform. This assumption is important for stability as e.g. discussed in \cite{Welper2015}.

\begin{proposition}
  \label{prop:tsi-tensor-error}
  Assume that there is a transform satisfying \eqref{eq:jump-align} and \eqref{eq:stab} and TSIs $\tsihp_n^1$ and $\tsi_n^2$ satisfying the error estimates \eqref{eq:rate-assumption}. Then
  \[
    \sup_{\hmu \in \hat{\param}} \|u(\cdot, \mu, \hmu) - (\tsihp_n^2 \circ \tsi_n^1 u)(\cdot, \mu, \hmu) \|_{L_1(\Omega)}  \le C \gamma \Lambda_n \rho^{-n^{1/2}} + C \rho^{-n^{1/(d-1)}}.
  \]
\end{proposition}

In order to prove the proposition, we need the following stability estimate.

\begin{lemma}
\label{lemma:stab}
Let all assumptions of Proposition \ref{prop:tsi-tensor-error} be satisfied and $u(\cdot, \mu, \hmu)$,  $v(\cdot, \mu, \hmu) \in L_1(\Omega)$ for all $\mu \in \param$ and $\hmu$ in $\hat{\param}$. Then, we have
\[
  \| (\tsi_n^2 u)(\cdot, \mu, \hmu) - (\tsi_n^2 v)(\cdot, \mu, \hmu) \|_{L_1(\Omega)} \le \gamma \Lambda_n^{d-1} \|u(\cdot, \phi_2, \hmu) - v(\cdot, \phi_2, \hmu)\|_{L_1(\Omega)},
\]
with $\phi_2 = \varphi(\hmu, \heta)(\mu)$.
\end{lemma}

\begin{proof}
Let $\hat{\param}_n \subset \hat{\param}$ be the interpolation points of $\tsi_n^2$. By the definition of the TSI, the definition of the Lebesgue constant and \eqref{eq:stab}, we have
\begin{multline*}
\| (\tsi_n^2 u)(\cdot, \mu, \hmu) - (\tsi_n^2 u)(\cdot, \mu, \hmu) \|_{L_1(\Omega)} \\
\begin{aligned}
   & = \left\| \sum_{\heta \in \hat{\param}_n} \ell_\heta(\hmu) \big[ u(\phi(\hmu, \heta)(x, \mu), \heta) - v(\phi(\hmu, \heta)(x, \mu), \heta) \big] \right\|_{L_1(\Omega)} \\
   & \le \left(\sup_{\hmu \in \hat{\param}} \sum_{\heta \in \hat{\param}_n} |\ell_\heta(\hmu)| \right) \max_{\heta \in \hat{\param}_n} \left\| u(\phi(\hmu, \heta)(x, \mu), \heta) - v(\phi(\hmu, \heta)(x, \mu), \heta) \right\|_{L_1(\Omega)} \\
\end{aligned}
\end{multline*}
Abbreviating $\phi(\hmu, \heta)(x,\mu) = [\phi_1(x), \phi_2]$, where we have used that by \eqref{eq:transform2} the second component does not depend on $x$, we this yields
\begin{multline*}
\| (\tsi_n^2 u)(\cdot, \mu, \hmu) - (\tsi_n^2 u)(\cdot, \mu, \hmu) \|_{L_1(\Omega)} \\
\begin{aligned}
   & \le \Lambda_n \max_{\heta \in \hat{\param}_n} \int_\Omega |u(y, \phi_2, \heta) - v(y, \phi_2, \heta)| \, d \phi_1{}_* \lambda(y) \\
    & \le \gamma \Lambda_n \max_{\heta \in \hat{\param}_n} \int_\Omega | u(y, \phi_2, \heta) - v(y, \phi_2, \heta)| \, dy,
\end{aligned}
\end{multline*}

which completes the proof.

\end{proof}

\begin{proof}[Proof of Proposition \ref{prop:tsi-tensor-error}]

For all $\mu$ and $\hmu$, we have
\begin{multline*}
  \|u(\cdot, \mu, \hmu) - (\tsi_n^2 \circ \tsi_n^1 v)(\cdot, \mu, \hmu) \|_{L_1(\Omega)} \\
  \begin{aligned}
  & \le \|u(\cdot, \mu, \hmu) - (\tsi_n^2 u)(\cdot, \mu, \hmu) \|_{L_1(\Omega)} \\
  & \quad + \|(\tsi_n^2 u)(\cdot, \mu, \hmu) - (\tsi_n^2 \circ \tsi_n^1 u)(\cdot, \mu, \hmu) \|_{L_1(\Omega)}
  \end{aligned}
\end{multline*}
The first term on the right hand side can be estimated by assumption \eqref{eq:rate-assumption}:
\[
  \|u(\cdot, \mu, \hmu) - (\tsi_n^2 u)(\cdot, \mu, \hmu) \|_{L_1(\Omega)} \le C \rho^{-n^{1/(d-1)}}.
\]
For the second term, we use Lemma \ref{lemma:stab} and the convergence assumption \eqref{eq:rate-assumption} to obtain
\begin{multline*}
  \|(\tsi_n^2 u)(\cdot, \mu, \hmu) - (\tsi_n^2 \circ \tsi_n^1 u)(\cdot, \mu, \hmu) \|_{L_1(\Omega)} \\
   \le \gamma \Lambda_n \|u(\cdot, \phi_2, \hmu) - (\tsi_n^1 u)(\cdot, \phi_2, \hmu) \|_{L_1(\Omega)} \le C \gamma \Lambda_n  \rho^{-n^{1/2}},
  \end{multline*}
  where $\phi_2 = \varphi(\hmu, \heta)(\mu)$. Combining these estimates completes the proof.
\end{proof}

In conclusion, we see that although we need to resolve the lower dimensional manifold where the jump set topology changes by some $h$ refinement, we can in principle achieve exponential convergence rates. 

\section{Numerical Experiments}
\label{sec:experiments}

This section contains some numerical experiments for the Burgers equation and a compressible flow over a forward facing step. The implementation uses Deal.II \cite{BangerthHartmannKanschat2007} and OpenFoam \cite{OpenFoam2017} for the calculation of the snapshots and Tensorflow \cite{AbadiAgarwalBarhamEtAl2015} for the optimization of the transforms.

\subsection{Burgers' Equation - Shock Collision}
\label{sec:burgers-two-shocks}

As a first numerical test, we consider the Burgers equation 
\begin{align}
  u_t + \frac{1}{2}\left(u^2\right)_x & = 0, & u(0,x) & = \left\{ \begin{array}{ll} \mu & x< 0 \\ \frac{\mu}{2} & 0 \le x < 1 \\ 0 & x \ge 0 \end{array} \right.,
  \label{eq:burgers-two-shocks}
\end{align}
with two initial shocks that collide at time $t=\mu/2$. Although $t$ is usually a physical variable, we consider both $\mu$ and the time $t$ as parameters. Since the solutions are piecewise constant, we only consider $h$ refinement and low interpolation orders. 

In order for the gradient descent optimizer to properly find the jumps, we use a spacial mesh of step size $0.01$ and smooth the solutions $u(x,t,\mu)$ so that the jumps become smooth with a width of about two times the mesh size. This way, the derivatives have sharp gradients that are visible to the spacial quadrature and emulate the Dirac deltas that appear in the correct derivative. For the transforms we use polynomials of degree one in time and space and stop the $h$ refinement once the maximal cell error drops below $2 \cdot 0.01$.

Some numerical results are shown in Figures \ref{fig:two-shocks-fine} and \ref{fig:two-shocks-coarse}, for fine and coarse quadrature of the $L_1$-error as described after \eqref{eq:opt-tensor}, respectively. The errors are reported in Table \ref{table:results}. Note that the fine and coarse quadrature achieve similar errors. In addition, the error is already of an order of magnitude that is expected from snapshots on a grid with mesh-size $0.01$ and smoothed jump with roughly twice that size. This fact renders computational experiments with regard to convergence rates difficult: We wish to see exponential rates in the parameter but only achieve first order in the snapshots due to the jump singularities and uniform girds. Nonetheless, Figure \ref{fig:two-shocks-rates} shows some convergence rates in dependence of the number of snapshots. Indeed, we see that the error saturates at roughly the expected error of the snapshots. In Figures \ref{fig:two-shocks-fine} and \ref{fig:two-shocks-coarse}, we also observe that the $h$-adaptivity tends to refine for times that are a little smaller than the actual collision time. In view of the discussion in Section \ref{sec:singular} this is not completely unexpected: The cells adjacent to the one containing the shock topology change must be sufficiently refined in order to avoid singular transforms.

\begin{table}[htb]
  \begin{center}
  \begin{tabular}{c|c|c|c|c}

    Example & $\sharp$ Snapshots & $\sharp$ Snapshots & Error      & Error \\
            & (tsi)   & (all)              & (training) & (sample) \\
    \hline
    \hline
    
    \eqref{eq:burgers-two-shocks}, fine quadrature & 
    \begin{filecontents*}{pics/two_shocks_data.csv}
SnapshotsTsi, SnapshotsAll, ErrorTrain, ErrorSample
14, 41, 1.81e-02, 6.15e-03
\end{filecontents*}
\csvreader[head to column names]{pics/two_shocks_data.csv}{}{\SnapshotsTsi & \SnapshotsAll & \ErrorTrain & \ErrorSample}
    
    \\

    \eqref{eq:burgers-two-shocks}, coarse quadrature & 
    \begin{filecontents*}{pics/two_shocks_coarse_quad_data.csv}
SnapshotsTsi, SnapshotsAll, ErrorTrain, ErrorSample
14, 31, 1.81e-02, 6.60e-03
\end{filecontents*}
\csvreader[head to column names]{pics/two_shocks_coarse_quad_data.csv}{}{\SnapshotsTsi & \SnapshotsAll & \ErrorTrain & \ErrorSample}

    \\
    
    \eqref{eq:burgers-shock-rwave} & 
    \begin{filecontents*}{pics/shock_rwave_data.csv}
SnapshotsTsi, SnapshotsAll, ErrorTrain, ErrorSample
21, 57, 2.45e-02, 2.85e-02
\end{filecontents*}
\csvreader[head to column names]{pics/shock_rwave_data.csv}{}{\SnapshotsTsi & \SnapshotsAll & \ErrorTrain & \ErrorSample}

    \\
    
    Section \ref{sec:ffstep} & 
    \begin{filecontents*}{pics/ffstep_data.csv}
SnapshotsTsi, SnapshotsAll, ErrorTrain, ErrorSample
12, 27, 4.64e-02, 5.09e-02
\end{filecontents*}
\csvreader[head to column names]{pics/ffstep_data.csv}{}{\SnapshotsTsi & \SnapshotsAll & \ErrorTrain & \ErrorSample}

  \end{tabular}
  \end{center}
  
  \caption{Results of various numerical experiments. The training error is the error calculated by the optimizer and the sample error is the maximal error at the parameters that are plotted in Figures \ref{fig:two-shocks-fine}, \ref{fig:two-shocks-coarse}, \ref{fig:shock-rwave} and \ref{fig:ffstep-snapshots}.}
  
  \label{table:results}
\end{table}

\pgfplotsset{axis function/.style={
  height=5cm, 
  yticklabel pos=right, 
  ticklabel style={font=\tiny}, 
  y label style={at={(axis description cs:0.2,0.5)}}
}}

\pgfplotsset{axis scatter/.style = {
  scatter/classes={
      train={mark size=1pt, color=red},
      nodes={mark=square*, mark size=1pt, color=blue},
      experiment={mark=diamond*, color=brown},
      collision={mark=-, thick, color=black}
    }
}}

\newcommand{\AddSnapshot}[2]{
      \addplot[mark=none, color=red, thick] table[x=x, y=#1, col sep=comma] {#2};
}

\newcommand{\AddTsi}[2]{
    \addplot[mark=none, color=blue, thick, dashed] table[x=x, y=#1, col sep=comma] {#2};
}

\newcommand{\AddNodes}[3]{
  \addplot[scatter, only marks, scatter src=explicit symbolic] table[x=#1, y=#2, meta=label, col sep=comma] {#3};  
}

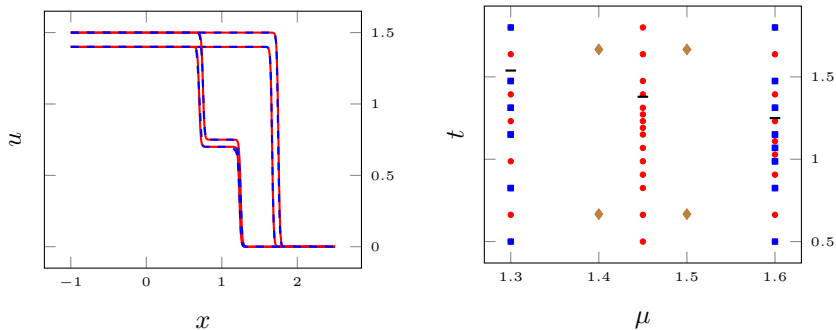
\begin{figure}[htb]

  \hfill
  \begin{tikzpicture}
    \begin{axis}[axis function, xlabel=$x$, ylabel=$u$] 

      \AddSnapshot{true 0}{pics/two_shocks_2d.csv}
      \AddTsi{tsi 0}{pics/two_shocks_2d.csv}

      \AddSnapshot{true 1}{pics/two_shocks_2d.csv}
      \AddTsi{tsi 1}{pics/two_shocks_2d.csv}

      \AddSnapshot{true 2}{pics/two_shocks_2d.csv}
      \AddTsi{tsi 2}{pics/two_shocks_2d.csv}

      \AddSnapshot{true 3}{pics/two_shocks_2d.csv}
      \AddTsi{tsi 3}{pics/two_shocks_2d.csv}

    \end{axis}
  \end{tikzpicture}
  \hfill
  \begin{tikzpicture}
    \begin{axis}[axis function, axis scatter, xlabel=$\mu$, ylabel=$t$] 

      \AddNodes{mu}{time}{pics/two_shocks_params.csv}

      \newcommand{\CollisionTime}[2]{2.0 / (#1 - #2)}
      \addplot[scatter, only marks, scatter src=explicit symbolic] 
        table[meta=label, 
          x=mu,
          y expr={\CollisionTime{\thisrowno{0}}{0}}
        ] {
          mu label
          1.3  collision
          1.45 collision
          1.6  collision
        };
      
    \end{axis}
  \end{tikzpicture}
  \hfill~

  \caption{Left: Blue dashed line: Component-wise $h$-adaptive TSI of \eqref{eq:burgers-two-shocks}. The quadrature for calculating the $L_1$-error uses the same mesh size $0.01$ in the physical and time direction, see the discussion after \eqref{eq:opt-tensor}. Red solid line: Corresponding correct solutions of \eqref{eq:burgers-two-shocks}. Right: Location of the snapshots. Blue dots: snapshots used in the TSI reconstruction. Red dots: snapshots used for training only. Brown diamonds: Parameters of the plots in the left figure. Black hyphen: time of the shock collision.}
  
  \label{fig:two-shocks-fine}
\end{figure}

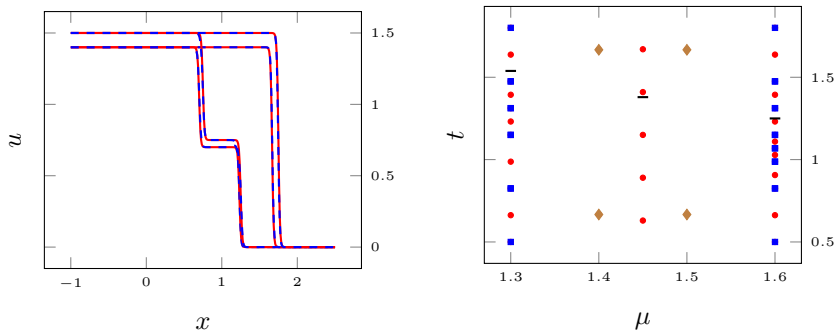
\begin{figure}[htb]

  \hfill
  \begin{tikzpicture}
  
      \begin{axis}[axis function, xlabel=$x$, ylabel=$u$]

      \addplot[mark=none, color=red, thick] table[x=x, y=true 0, col sep=comma] {pics/two_shocks_2d.csv};
      \addplot[mark=none, color=blue, thick, dashed] table[x=x, y=tsi 0, col sep=comma] {pics/two_shocks_coarse_quad_2d.csv};

      \addplot[mark=none, color=red, thick] table[x=x, y=true 1, col sep=comma] {pics/two_shocks_2d.csv};
      \addplot[mark=none, color=blue, thick, dashed] table[x=x, y=tsi 1, col sep=comma] {pics/two_shocks_coarse_quad_2d.csv};

      \addplot[mark=none, color=red, thick] table[x=x, y=true 2, col sep=comma] {pics/two_shocks_2d.csv};
      \addplot[mark=none, color=blue, thick, dashed] table[x=x, y=tsi 2, col sep=comma] {pics/two_shocks_coarse_quad_2d.csv};

      \addplot[mark=none, color=red, thick] table[x=x, y=true 3, col sep=comma] {pics/two_shocks_2d.csv};
      \addplot[mark=none, color=blue, thick, dashed] table[x=x, y=tsi 3, col sep=comma] {pics/two_shocks_coarse_quad_2d.csv};

    \end{axis}
  \end{tikzpicture}
  \hfill
  \begin{tikzpicture}
  
    \begin{axis}[axis function, axis scatter, xlabel=$\mu$, ylabel=$t$]

      \AddNodes{mu}{time}{pics/two_shocks_coarse_quad_params.csv}

      \newcommand{\CollisionTime}[2]{2.0 / (#1 - #2)}
      \addplot[scatter, only marks, scatter src=explicit symbolic] 
        table[meta=label, 
          x=mu,
          y expr={\CollisionTime{\thisrowno{0}}{0}}
        ] {
          mu label
          1.3  collision
          1.45 collision
          1.6  collision
        };
     
    \end{axis}
  \end{tikzpicture}
  \hfill~

  \caption{Same as in Figure \ref{fig:two-shocks-fine}, but with a coarse quadrature for the $L_1$-error in $t$ direction, as described after \eqref{eq:opt-tensor}.}

  \label{fig:two-shocks-coarse}

\end{figure}

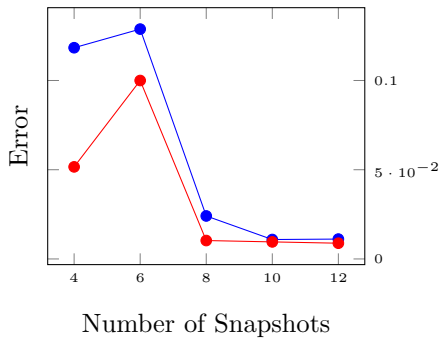
\begin{figure}[htb]

  \begin{center}
  \begin{tikzpicture}
    \begin{axis}[axis function, xlabel=Number of Snapshots, ylabel=Error] 

      \addplot[color=blue, mark=*] table[x=n_snapshots, y=error, col sep=comma] {pics/two_shocks_rate.csv};

      \addplot[color=red, mark=*] table[x=n_snapshots, y=error, col sep=comma] {pics/two_shocks_rate_coarse_quad.csv};

    \end{axis}
  \end{tikzpicture}
  \end{center}

  \caption{Errors for example \eqref{eq:burgers-two-shocks} with fine quadrature.}

  \label{fig:two-shocks-rates}

\end{figure}

\subsection{Burgers' Equation - Shock and Rarefaction Wave}

For a second numerical test, we use the same setup as in Section \ref{sec:burgers-two-shocks}, but with different parameter:
\begin{align}
u_t + \frac{1}{2}\left(u^2\right)_x & = 0, & u(0,x) & = \left\{ \begin{array}{ll} 1.5 & x< 0 \\ 0 & 0 \le x < 1 \\ \mu & x \ge 0 \end{array} \right., &
\begin{aligned}
  0 & \le t \le 2 \\
  -0.5 & \le \mu \le 0.5.
\end{aligned}
\label{eq:burgers-shock-rwave}
\end{align}
Again, we consider time $t$ and $\mu$ as parameters. The initial value has two jumps. The left yields a shock and the right a shock or a rarefaction wave, depending on the value of $\mu$. In any case, the shock or rarefaction wave collide with the left shock within the given time interval. This leaves us with a few changes in the topology of the jumps and kinks: First the collision, second the change form shock to rarefaction wave and third the transition from the left initial jump to a rarefaction wave in case $\mu>0$. 

We use TSIs of degree two and three for $\mu$ and $t$, respectively and polynomials of degree two for the spacial resolution of the transforms and the coarse quadrature described after \eqref{eq:opt-tensor}. The numerical errors are shown in Table \ref{table:results} and some plots are shown in figures \ref{fig:shock-rwave} and \ref{fig:shock-rwave-pcolor}. We see some local refinements for $\mu=-0.5$ at the time of the shock collision and for $\mu = 0.5$ at $t=0$, where the jump in the initial value becomes a rarefaction wave with two kinks. The collision of the rarefaction wave with the shock has little refinement. Nonetheless, the overall error is already close to the error that we expect from snapshots of a spacial resolution of $0.01$, so that further refinements are questionable. 

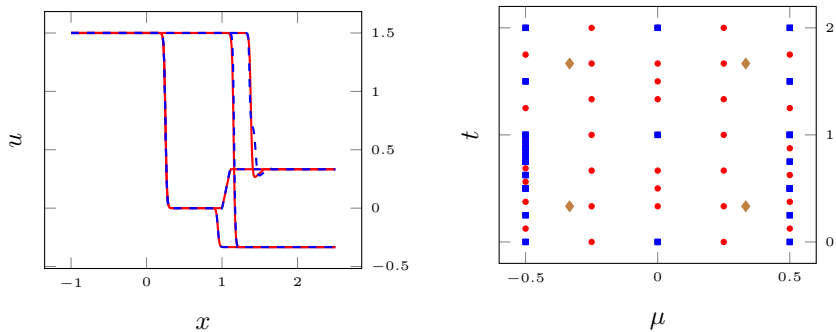
\begin{figure}[htb]
    
    \hfill
    \begin{tikzpicture}
    \begin{axis}[axis function, xlabel=$x$, ylabel=$u$] 
    
      \AddSnapshot{true 0}{pics/shock_rwave_2d.csv}
      \AddTsi{tsi 0}{pics/shock_rwave_2d.csv}
    
      \AddSnapshot{true 1}{pics/shock_rwave_2d.csv}
      \AddTsi{tsi 1}{pics/shock_rwave_2d.csv}
    
      \AddSnapshot{true 2}{pics/shock_rwave_2d.csv}
      \AddTsi{tsi 2}{pics/shock_rwave_2d.csv}
    
      \AddSnapshot{true 3}{pics/shock_rwave_2d.csv}
      \AddTsi{tsi 3}{pics/shock_rwave_2d.csv}
    
    \end{axis}
    \end{tikzpicture}
    \hfill
    \begin{tikzpicture}
    \begin{axis}[axis function, axis scatter, xlabel=$\mu$, ylabel=$t$] 
    
      \AddNodes{mu}{time}{pics/shock_rwave_params.csv}
        
    \end{axis}
    \end{tikzpicture}
    \hfill~
    
    \caption{Same as in Figure \ref{fig:two-shocks-fine}, but for the example \eqref{eq:burgers-shock-rwave}.}
    
    \label{fig:shock-rwave}
\end{figure}

\begin{figure}[htb]
    
    \pgfplotsset{plot graphics/burgers bounds/.style={
            xmin=-1, xmax=2.5, ymin=0.5, ymax=1.8
        }}
        
    \hfill
    \begin{tikzpicture}
    \begin{axis}[axis function, xlabel=$x$, ylabel=$t$]
    
    \addplot graphics[burgers bounds] {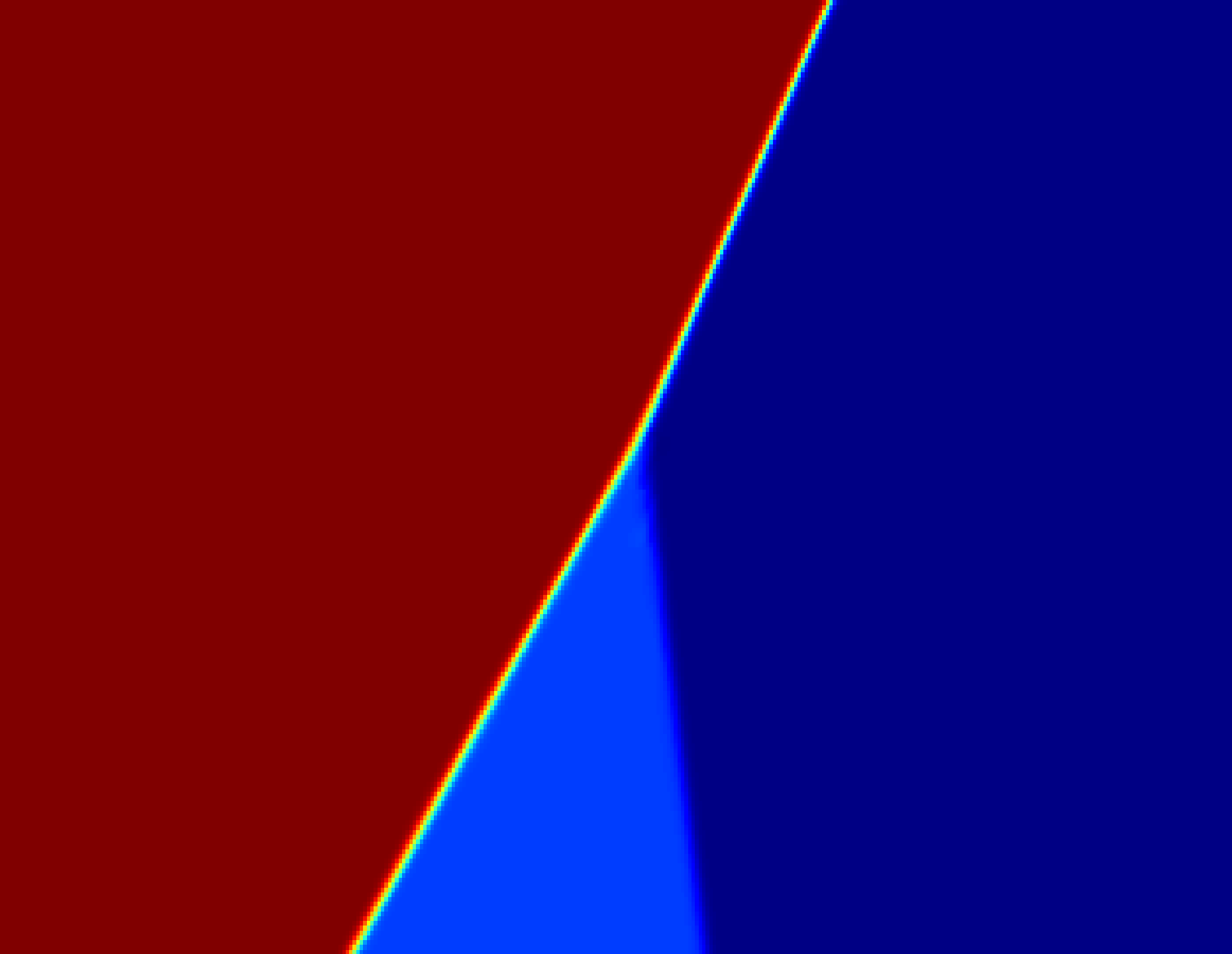};
    
    \end{axis}
    \end{tikzpicture}
    \hfill
    \begin{tikzpicture}
    \begin{axis}[axis function, xlabel=$x$, ylabel=$t$] 

    \addplot graphics[burgers bounds] {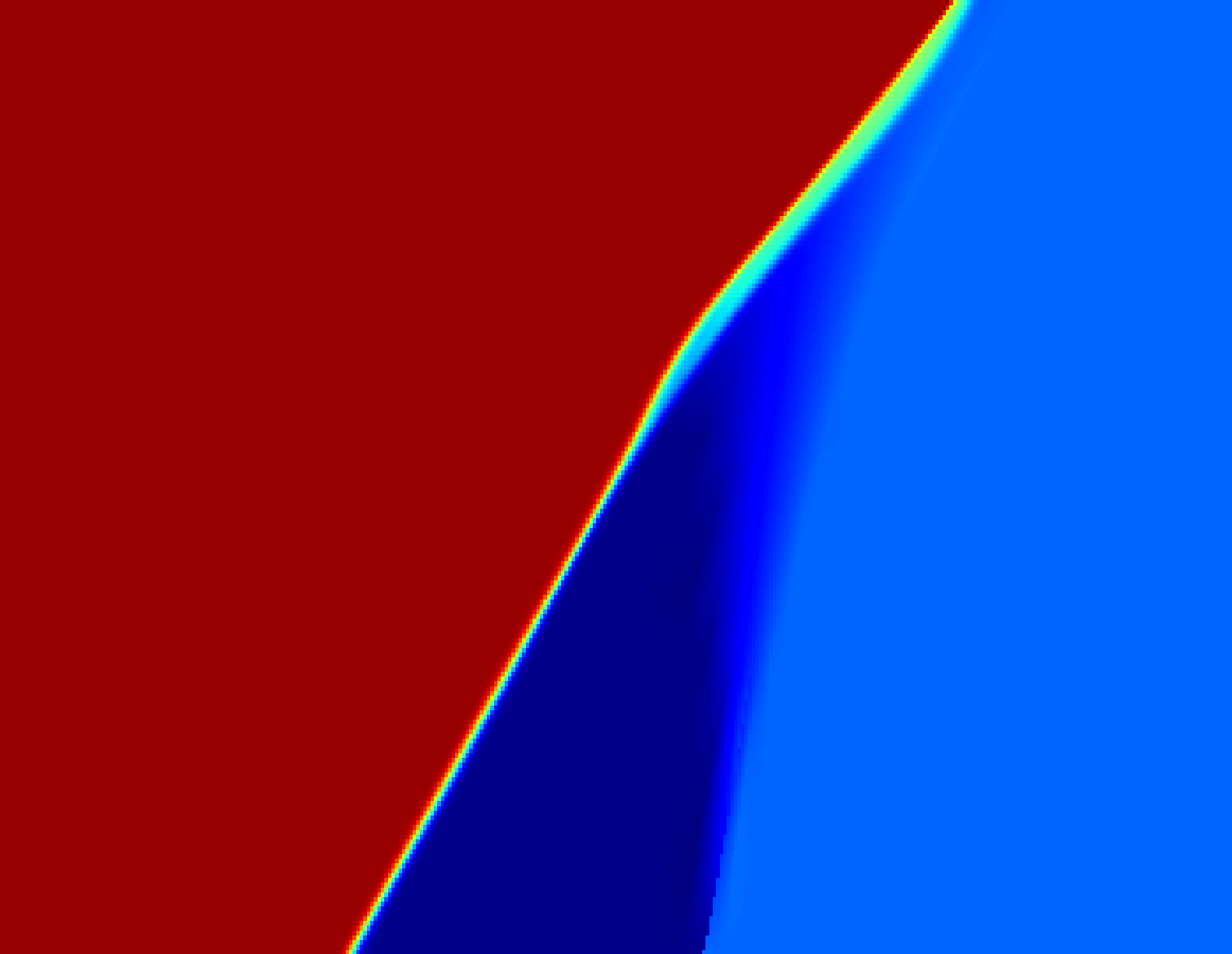};
        
    \end{axis}
    \end{tikzpicture}
    \hfill~

    \caption{Solutions of example \eqref{eq:burgers-shock-rwave} for the two values of $\mu$ given by the brown diamonds in Figure \ref{fig:shock-rwave} right.}
    
    \label{fig:shock-rwave-pcolor}
\end{figure}

\section{Forward Facing Step}
\label{sec:ffstep}

For a final test, we consider a compressible flow over a forward facing step from the OpenFoam tutorial \cite{OpenFoam2017}, with a mesh size of $h=0.025$. The first parameter is the inflow Mach number and the second, again for simplicity, time.

The transforms are of the form $\phi(\mu, \eta)(x) = x + \varphi(\mu, \eta)(x)$, where $\varphi(\mu, \eta)$ are finite elements of degree two on a initial coarse mesh. In order to ensure that the transforms $\phi(\mu, \eta)$ map $\Omega$ to itself, we us the boundary condition
\[
\begin{aligned}
  \varphi(\mu, \eta)(x) \cdot \nu & = 0, & x & \in \partial \Omega,
\end{aligned}
\]
where $\nu$ is the outward normal. For sufficiently small $\varphi(\mu, \eta)$ this condition ensures that the transforms are homeomorphisms of $\Omega$ and allows the transform to slide points along the boundary.

The errors for TSI degree two are reported in Table \ref{table:results}, the snapshots and training snapshots in Figure \ref{fig:ffstep-snapshots} and some plots are shown in Figure \ref{fig:ffstep-plots}. The plot in Figure \ref{fig:ffstep-plots} show no signs of a stair-casing behavior and the errors are about two times the mesh-size of the spacial variables, which seems reasonable. With increasing Mach number, the flow develops an additional jump, as can be seen in the plots. The Mach number of this transition is roughly captured by the locations of the highest refinement in Figure \ref{fig:ffstep-snapshots}.

\begin{figure}[htb]

  \begin{center}
  \begin{tikzpicture}
  
    \begin{axis}[axis function, axis scatter, xlabel=$\mu$, ylabel=$t$]

      \AddNodes{mu}{time}{pics/ffstep_params.csv}

    \end{axis}
  \end{tikzpicture}
  \end{center}
  
  \caption{Same as in Figure \ref{fig:two-shocks-fine}, for the forward facing step example, Section \ref{sec:ffstep}.}

  \label{fig:ffstep-snapshots}

\end{figure}
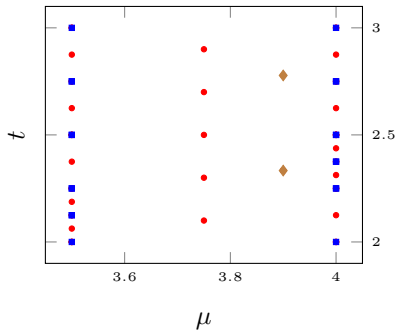

\begin{figure}
  \subfloat[][TSI \centering $\mu=2.33$]{
    \includegraphics[width=0.2\textwidth]{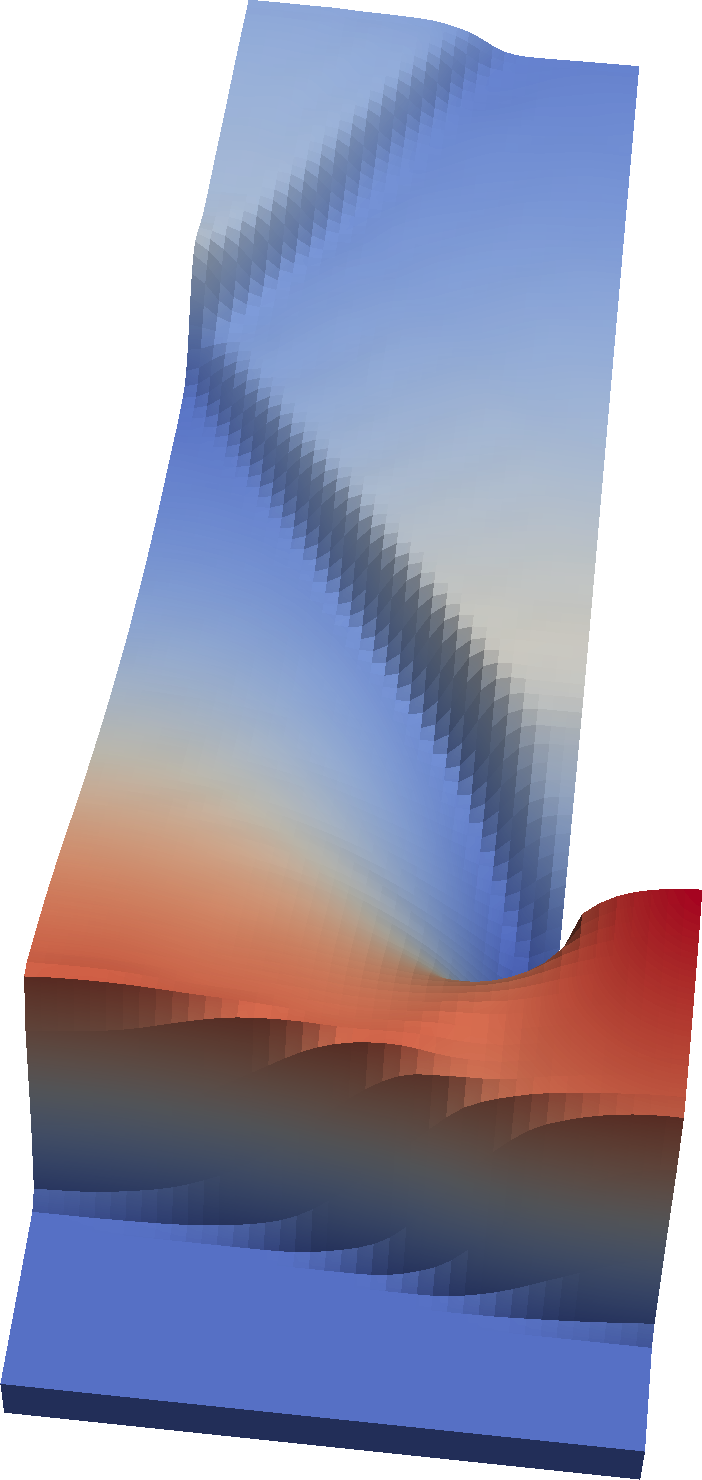}
  }
  \hfill
  \subfloat[][true \centering $\mu=2.33$]{
    \includegraphics[width=0.2\textwidth]{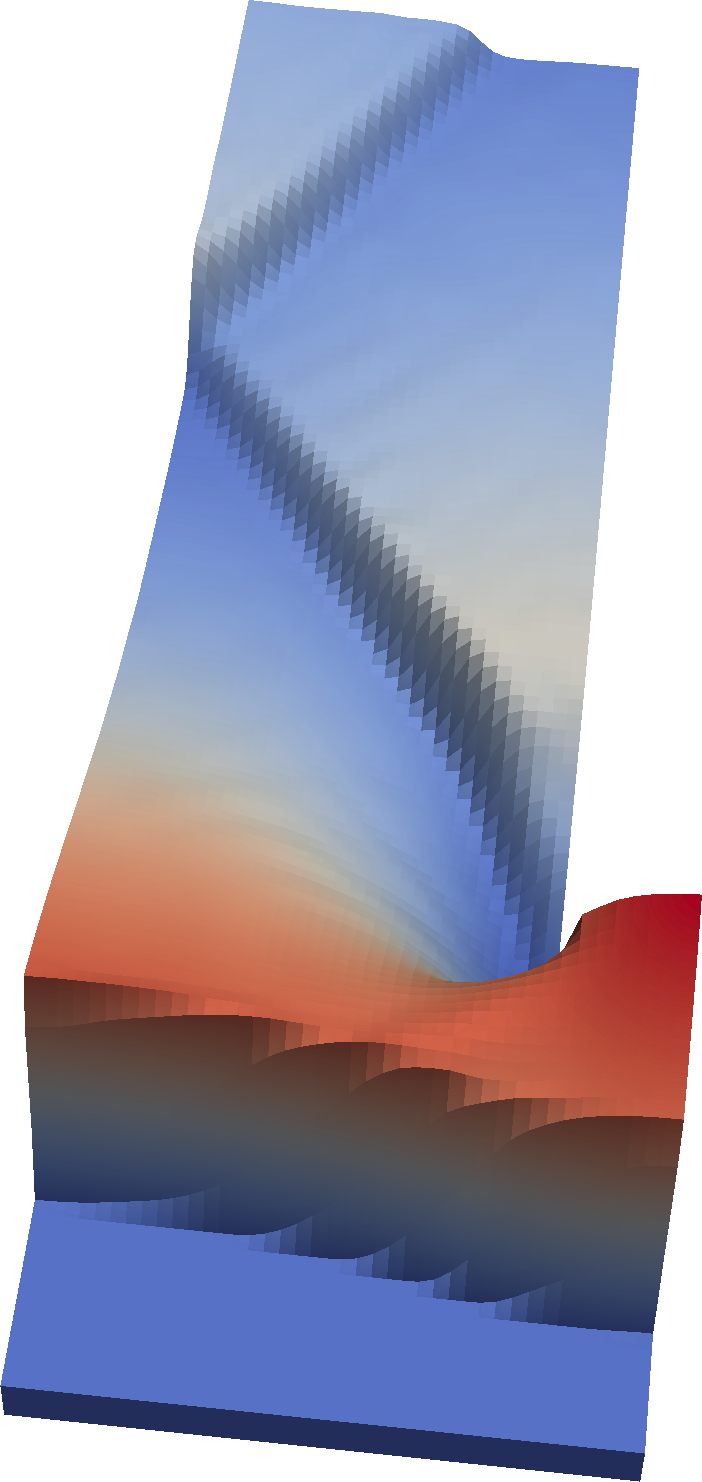}
  }
  \hfill
  \subfloat[][TSI \centering $\mu=2.77$]{
    \includegraphics[width=0.2\textwidth]{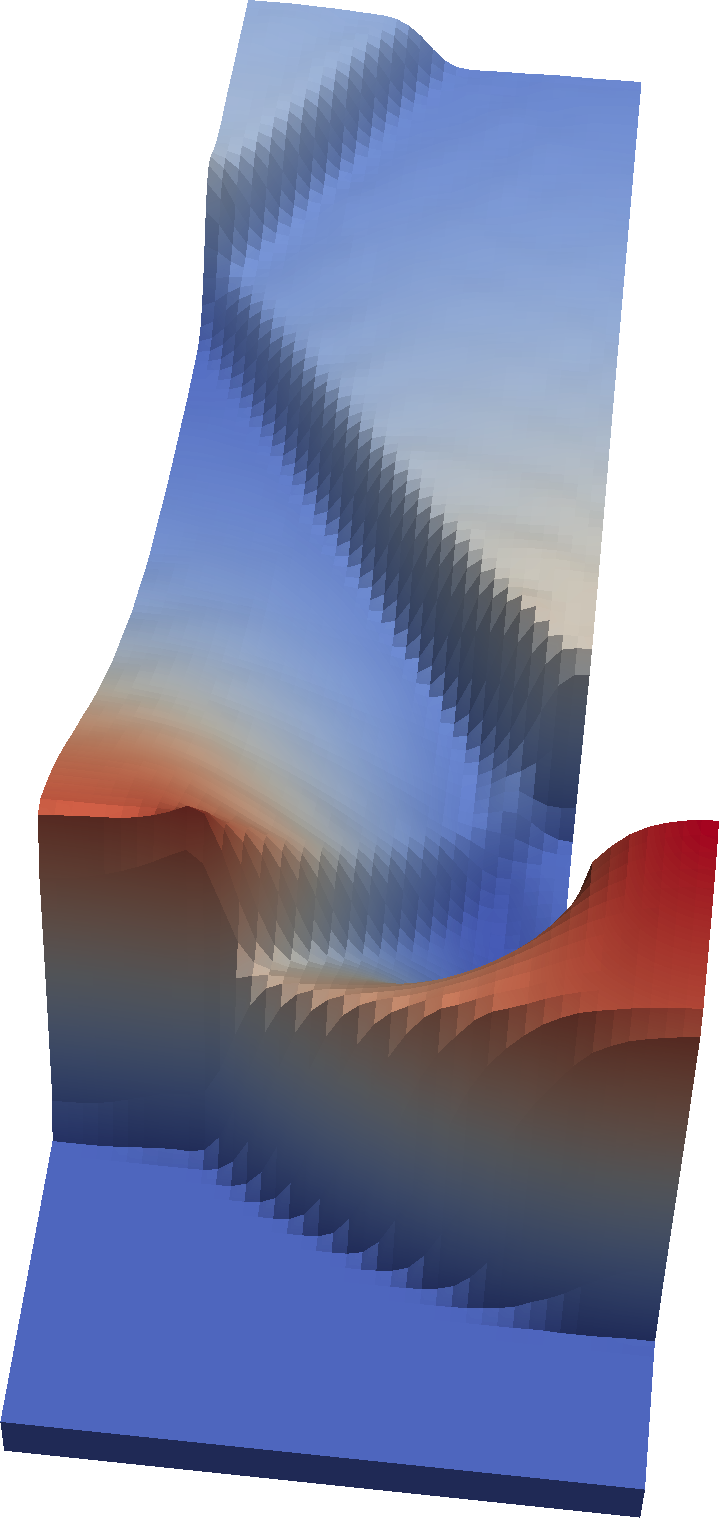}
  }
  \hfill
  \subfloat[][true  \centering $\mu=2.77$]{
    \includegraphics[width=0.2\textwidth]{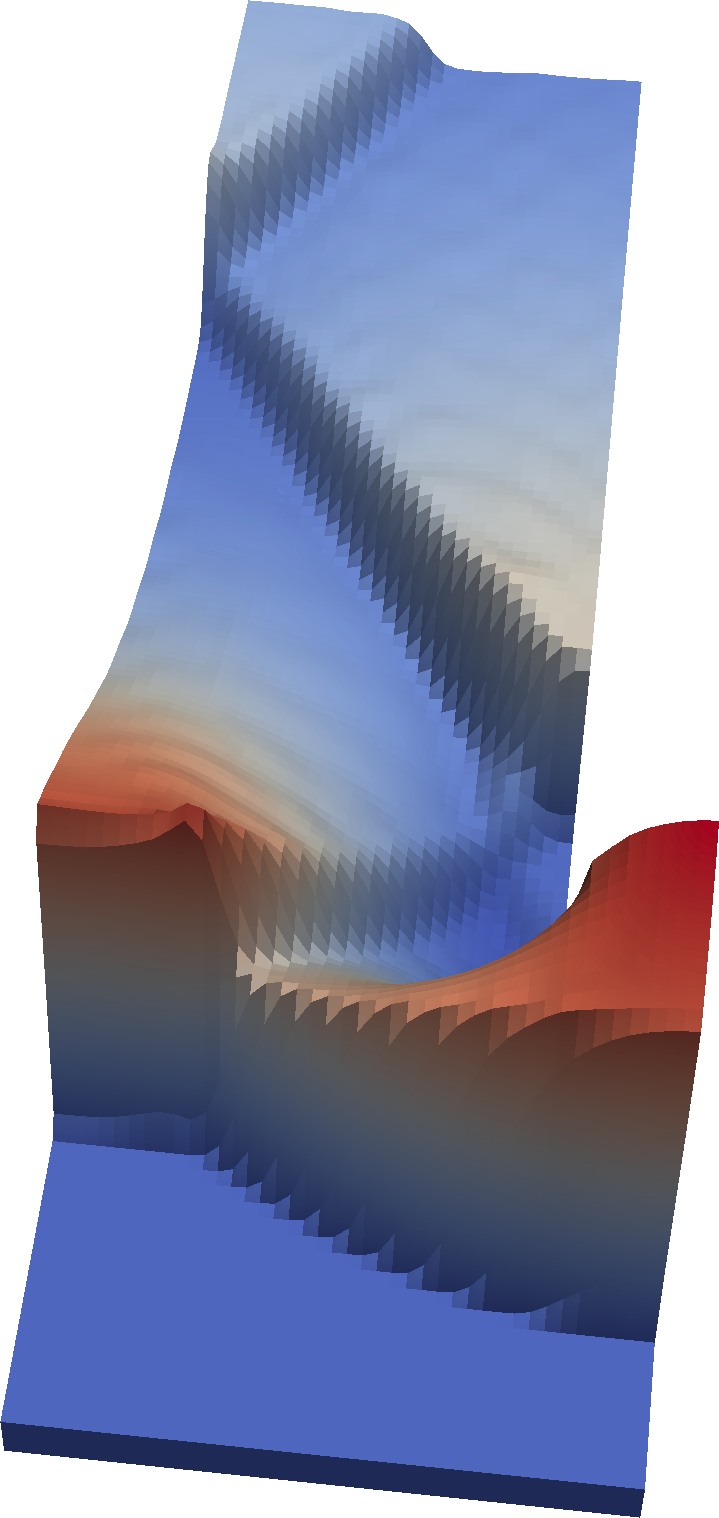}
  }
  
  \caption{Solutions of the forward facing step example.}
  
  \label{fig:ffstep-plots}
\end{figure}

\section{Conclusion}

The major assumption underlying the TSI is that the jump sets of the snapshots and of the target function $u(\cdot, \mu)$ can be aligned. For many problems this is overly restrictive. However, the changes in the shock topology structure are typically local in parameter space, so that we use $h$ and $hp$ adaptive methods to resolve them. Since local refinements scale badly in high dimensions, we only apply this $hp$-refinement in one parameter dimension and use a ``tensor product'' like construction for higher parameter dimensions, based on the ability of the TSI to align singularities with coordinate axes. We have proven that for some simple model cases this algorithm achieves exponential convergence rates and tested it with numerical experiments. 

Comparable to full tensor products, this construction alone does not avoid the curse of dimensionality, but is a possible starting point for methods that do. An investigation is left for future research. The method presented in this article can successfully deal with a limited number of jump set topology changes, as found e.g. in many steady state problems, but scales badly if they are numerous. However, these changes are not only local in parameter, but also local in space. This is expected to improve the situation and left for future research.

\bibliographystyle{plain}
\bibliography{tsi-hp}

\begin{thebibliography}{10}

\bibitem{OpenFoam2017}
Openfoam: Tutorial guide, 2017.

\bibitem{AbadiAgarwalBarhamEtAl2015}
Mart\'{\i}n Abadi, Ashish Agarwal, Paul Barham, Eugene Brevdo, Zhifeng Chen,
  Craig Citro, Greg~S. Corrado, Andy Davis, Jeffrey Dean, Matthieu Devin,
  Sanjay Ghemawat, Ian Goodfellow, Andrew Harp, Geoffrey Irving, Michael Isard,
  Yangqing Jia, Rafal Jozefowicz, Lukasz Kaiser, Manjunath Kudlur, Josh
  Levenberg, Dan Man\'{e}, Rajat Monga, Sherry Moore, Derek Murray, Chris Olah,
  Mike Schuster, Jonathon Shlens, Benoit Steiner, Ilya Sutskever, Kunal Talwar,
  Paul Tucker, Vincent Vanhoucke, Vijay Vasudevan, Fernanda Vi\'{e}gas, Oriol
  Vinyals, Pete Warden, Martin Wattenberg, Martin Wicke, Yuan Yu, and Xiaoqiang
  Zheng.
\newblock {TensorFlow}: Large-scale machine learning on heterogeneous systems,
  2015.
\newblock Software available from tensorflow.org.

\bibitem{AbgrallAmsallem2015}
Remi Abgrall and David Amsallem.
\newblock Robust model reduction by $l_1$-norm minimization and approximation
  via dictionaries: Application to linear and nonlinear hyperbolic problems.
\newblock \url{http://arxiv.org/abs/1506.06178}, 2015.

\bibitem{BabuskaGuo1992}
I.~Babu\v{s}ka and B.Q. Guo.
\newblock The h, p and h-p version of the finite element method; basis theory
  and applications.
\newblock {\em Advances in Engineering Software}, 15(3):159 -- 174, 1992.

\bibitem{BabuskaNobileTempone2007}
Ivo Babu\v{s}ka, Fabio Nobile, and Ra\'{u}l Tempone.
\newblock A stochastic collocation method for elliptic partial differential
  equations with random input data.
\newblock {\em SIAM Journal on Numerical Analysis}, 45(3):1005--1034, 2007.

\bibitem{BabuskaTemponeZouraris2004}
Ivo Babu\v{s}ka, Ra\'{u}l Tempone, and Georgios~E. Zouraris.
\newblock Galerkin finite element approximations of stochastic elliptic partial
  differential equations.
\newblock {\em SIAM Journal on Numerical Analysis}, 42(2):800--825, 2004.

\bibitem{BangerthHartmannKanschat2007}
W.~Bangerth, R.~Hartmann, and G.~Kanschat.
\newblock Deal.ii -- a general-purpose object-oriented finite element library.
\newblock {\em ACM Trans. Math. Softw.}, 33(4), 2007.

\bibitem{BungartzGriebel2004}
Hans-Joachim Bungartz and Michael Griebel.
\newblock Sparse grids.
\newblock {\em Acta Numerica}, 13:147–269, May 2004.

\bibitem{CagniartMadayStamm2016}
Nicolas Cagniart, Yvon Maday, and Benjamin Stamm.
\newblock {Model Order Reduction for Problems with large Convection Effects}.
\newblock \url{http://hal.upmc.fr/hal-01395571}, 2016.

\bibitem{ChenGottliebHesthaven2005}
Qian-Yong Chen, David Gottlieb, and Jan~S. Hesthaven.
\newblock Uncertainty analysis for the steady-state flows in a dual throat
  nozzle.
\newblock {\em Journal of Computational Physics}, 204(1):378 -- 398, 2005.

\bibitem{CohenDeVore2015}
Albert Cohen and Ronald DeVore.
\newblock Approximation of high-dimensional parametric pdes.
\newblock {\em Acta Numerica}, 24:1--159, 005 2015.

\bibitem{CohenDeVoreSchwab2010}
Albert Cohen, Ronald DeVore, and Christoph Schwab.
\newblock Convergence rates of best n-term galerkin approximations for a class
  of elliptic spdes.
\newblock {\em Foundations of Computational Mathematics}, 10(6):615--646, 2010.

\bibitem{CohenDeVoreSchwab2011}
Albert Cohen, Ronald DeVore, and Christoph Schwab.
\newblock Analytic regularity and polynomial approximation of parametric and
  stochastic elliptic pdes.
\newblock {\em Analysis and Applications}, 09(01):11--47, 2011.

\bibitem{ConstantineIaccarino2012}
P.G. Constantine and G.~Iaccarino.
\newblock Reduced order models for parameterized hyperbolic conservations laws
  with shock reconstruction.
\newblock Technical report, Stanford Center for Turbulence Research Annual
  Research Briefs 2012, 2012.
\newblock
  \url{http://citeseerx.ist.psu.edu/viewdoc/download?doi=10.1.1.727.9918&rep=rep1&type=pdf}.

\bibitem{Dahmen2015}
Wolfgang Dahmen.
\newblock How to best sample a solution manifold?
\newblock In G{\"o}tz~E. Pfander, editor, {\em Sampling Theory, a Renaissance:
  Compressive Sensing and Other Developments}, pages 403--435. Springer
  International Publishing, 2015.

\bibitem{DahmenPleskenWelper2014}
Wolfgang Dahmen, Christian Plesken, and Gerrit Welper.
\newblock Double greedy algorithms: Reduced basis methods for transport
  dominated problems.
\newblock {\em ESAIM: Mathematical Modelling and Numerical Analysis},
  48:623--663, 5 2014.

\bibitem{Demkowicz2006}
Leszek Demkowicz.
\newblock {\em Computing with Hp-Adaptive Finite Elements, Vol. 1: One and Two
  Dimensional Elliptic and Maxwell Problems}, volume~1.
\newblock Chapman and Hall/CRC, 2006.

\bibitem{DespresPoeetteLucor2013}
Bruno Despr{\'e}s, Ga{\"e}l Po{\"e}tte, and Didier Lucor.
\newblock {\em Robust Uncertainty Propagation in Systems of Conservation Laws
  with the Entropy Closure Method}, pages 105--149.
\newblock Springer International Publishing, Cham, 2013.

\bibitem{EftangPateraRnquist2010}
Jens~L. Eftang, Anthony~T. Patera, and Einar~M. R{\o}nquist.
\newblock An ``$hp$'' certified reduced basis method for parametrized elliptic
  partial differential equations.
\newblock {\em SIAM Journal on Scientific Computing}, 32(6):3170--3200, 2010.

\bibitem{EftangStamm2012}
Jens~L. Eftang and Benjamin Stamm.
\newblock Parameter multi-domain `hp' empirical interpolation.
\newblock {\em International Journal for Numerical Methods in Engineering},
  90(4):412--428, 2012.

\bibitem{EvansGariepy2015}
Lawrence~Craig Evans and Ronald~F. Gariepy.
\newblock {\em Measure Theory and Fine Properties of Functions}.
\newblock CRC Press, 2015.

\bibitem{GassGlauMahlstedtMair2015}
Maximilian Gaß, Kathrin Glau, Mirco Mahlstedt, and Maximilian Mair.
\newblock Chebyshev interpolation for parametric option pricing.
\newblock \url{https://arxiv.org/abs/1505.04648}, 2015.

\bibitem{GerbeauLombardi2012}
Jean-Fr\'{e}d\'{e}ric Gerbeau and Damiano Lombardi.
\newblock Reduced-order modeling based on approximated lax pairs.
\newblock Technical report, INRIA, 2012.
\newblock \url{http://arxiv.org/pdf/1211.4153v1}.

\bibitem{GerbeauLombardi2014}
Jean-Fr\'{e}d\'{e}ric Gerbeau and Damiano Lombardi.
\newblock Approximated lax pairs for the reduced order integration of nonlinear
  evolution equations.
\newblock {\em Journal of Computational Physics}, 265:246 -- 269, 2014.

\bibitem{GunzburgerWebsterZhang2014}
Max~D. Gunzburger, Clayton~G. Webster, and Guannan Zhang.
\newblock Stochastic finite element methods for partial differential equations
  with random input data.
\newblock {\em Acta Numerica}, 23:521--650, 005 2014.

\bibitem{HaasdonkOhlberger2008a}
B.~Haasdonk and M.~Ohlberger.
\newblock Reduced basis method for explicit finite volume approximations of
  nonlinear conservation laws.
\newblock In {\em Procceedings of the 12th International Conference on
  Hyperbolic Problems: Theory, Numerics, Application}, College Park, Maryland,
  USA, June 09-13 2008.

\bibitem{HaasdonkOhlberger2008}
Bernard Haasdonk and Mario Ohlberger.
\newblock Reduced basis method for finite volume approximations of parametrized
  linear evolution equations.
\newblock {\em ESAIM: M2AN}, 42(2):277--302, 2008.

\bibitem{IaccarinoPetterssonNordstroemWitteveen2010}
Gianluca Iaccarino, Per Pettersson, Jan Nordstr\"{o}m, and Jeroen Witteveen.
\newblock Numerical methods for uncertainty propagation in high speed flows.
\newblock In J.~C.~F. Pereira and A.~Sequeira, editors, {\em V European
  Conference on Computational Fluid Dynamics ECCOMAS CFD 2010}, 2010.

\bibitem{JakemanNarayanXiu2013}
John~D. Jakeman, Akil Narayan, and Dongbin Xiu.
\newblock Minimal multi-element stochastic collocation for uncertainty
  quantification of discontinuous functions.
\newblock {\em Journal of Computational Physics}, 242(0):790 -- 808, 2013.

\bibitem{JinXiuZhu2016}
Shi Jin, Dongbin Xiu, and Xueyu Zhu.
\newblock A well-balanced stochastic galerkin method for scalar hyperbolic
  balance laws with random inputs.
\newblock {\em Journal of Scientific Computing}, 67(3):1198--1218, 2016.

\bibitem{KunischVolkwein2001}
K.~Kunisch and S.~Volkwein.
\newblock Galerkin proper orthogonal decomposition methods for parabolic
  problems.
\newblock {\em Numerische Mathematik}, 90(1):117--148, 2001.

\bibitem{KunischVolkwein2002}
K.~Kunisch and S.~Volkwein.
\newblock Galerkin proper orthogonal decomposition methods for a general
  equation in fluid dynamics.
\newblock {\em SIAM Journal on Numerical Analysis}, 40(2):492--515, 2002.

\bibitem{MishraSchwab2012}
S.~Mishra and Ch. Schwab.
\newblock Sparse tensor multi-level monte carlo finite volume methods for
  hyperbolic conservation laws with random initial data.
\newblock {\em Math. Comp.}, 81:1979--2018, 2012.

\bibitem{MitchellMcClain2014}
William~F. Mitchell and Marjorie~A. McClain.
\newblock A comparison of hp-adaptive strategies for elliptic partial
  differential equations.
\newblock {\em ACM Trans. Math. Softw.}, 41(1):2:1--2:39, October 2014.

\bibitem{NguyenRozzaPatera2009}
Ngoc-Cuong Nguyen, Gianluigi Rozza, and AnthonyT. Patera.
\newblock Reduced basis approximation and a posteriori error estimation for the
  time-dependent viscous burgers’ equation.
\newblock {\em Calcolo}, 46(3):157--185, 2009.

\bibitem{OhlbergerRave2013}
Mario Ohlberger and Stephan Rave.
\newblock Nonlinear reduced basis approximation of parameterized evolution
  equations via the method of freezing.
\newblock {\em Comptes Rendus Mathematique}, 351(23–24):901 -- 906, 2013.

\bibitem{OhlbergerRave2016}
Mario Ohlberger and Stephan Rave.
\newblock Reduced basis methods: Success, limitations and future challenges.
\newblock In D.~Handlovi\v{c}ov\'{a}, A.~\v{S}ev\v{c}ovi\v{c}, editor, {\em
  Proceedings of the Conference Algoritmy 2016}, pages 1--12. Publishing House
  of Slovak University of Technology in Bratislava, 2016.

\bibitem{PacciariniRozza2014}
Paolo Pacciarini and Gianluigi Rozza.
\newblock Stabilized reduced basis method for parametrized advection--diffusion
  {PDEs}.
\newblock {\em Computer Methods in Applied Mechanics and Engineering}, 274(0):1
  -- 18, 2014.

\bibitem{PateraRozza2006}
A.T. Patera and G.~Rozza.
\newblock Reduced basis approximation and a posteriori error estimation for
  parametrized partial differential equations.
\newblock Version 1.0, Copyright MIT 2006–2007, to appear in (tentative
  rubric) MIT Pappalardo Graduate Monographs in Mechanical Engineering,
  2006–2007.

\bibitem{PetterssonAbbasbIaccarinoEtAl2010}
Per Pettersson, Qaisar Abbasb, Gianluca Iaccarino, and Jan Nordstr\"{o}m.
\newblock Efficiency of shock capturing schemes for burgers’ equation with
  boundary uncertainty.
\newblock In {\em Seventh South African Conference on Computational and Applied
  Mechanics SACAM10}, 2010.

\bibitem{PetterssonIaccarinoNordstroem2014}
Per Pettersson, Gianluca Iaccarino, and Jan Nordstr\"{o}m.
\newblock A stochastic galerkin method for the euler equations with roe
  variable transformation.
\newblock {\em Journal of Computational Physics}, 257, Part A:481 -- 500, 2014.

\bibitem{PulchXiu2012}
Roland Pulch and Dongbin Xiu.
\newblock Generalised polynomial chaos for a class of linear conservation laws.
\newblock {\em Journal of Scientific Computing}, 51(2):293--312, 2012.

\bibitem{ReissSchulzeSesterhenn2015}
Julius Reiss, Philipp Schulze, and J\"{o}rn Sesterhenn.
\newblock The shifted proper orthogonal decomposition: A mode decomposition for
  multiple transport phenomena.
\newblock \url{https://arxiv.org/abs/1512.01985}, 2015.

\bibitem{RimMoeLeVeque2017}
Donsub Rim, Scott Moe, and Randall~J LeVeque.
\newblock Transport reversal for model reduction of hyperbolic partial
  differential equations.
\newblock \url{https://arxiv.org/abs/1701.07529}, 2017.

\bibitem{RozzaHuynhPatera2008}
G.~Rozza, D.B.P. Huynh, and A.T. Patera.
\newblock Reduced basis approximation and a posteriori error estimation for
  affinely parametrized elliptic coercive partial differential equations.
\newblock {\em Archives of Computational Methods in Engineering},
  15(3):229--275, 2008.

\bibitem{Schwab1999}
Christoph Schwab.
\newblock {\em p- and hp- Finite Element Methods: Theory and Applications to
  Solid and Fluid Mechanics}.
\newblock Numerical Mathematics and Scientific Computation. Clarendon Press, 1
  edition, 1999.

\bibitem{SenVeroyHuynhEtAl2006}
S.~Sen, K.~Veroy, D.B.P. Huynh, S.~Deparis, N.C. Nguyen, and A.T. Patera.
\newblock ``natural norm'' a posteriori error estimators for reduced basis
  approximations.
\newblock {\em Journal of Computational Physics}, 217(1):37 -- 62, 2006.

\bibitem{Sirovich1987}
Lawrence Sirovich.
\newblock Turbulence and the dynamics of coherent structures, parts {I}-{III}.
\newblock {\em Quart. Appl. Math.}, 45(3):561--590, 1987.

\bibitem{SueliHoustonSchwab2000}
E.~S\"{u}li, P.~Houston, and C.~Schwab.
\newblock $hp$-finite element methods for hyperbolic problem.
\newblock In J.~R. Whiteman, editor, {\em The mathematics of finite elements
  and applications X, MAFELAP 199}, pages 143--162. Elsevier, Amsterdam, 2000.
\newblock Proceedings of the 10th conference, Brunel Univ., Uxbridge,
  Middlesex, GB, June 22-25, 19.

\bibitem{TaddeiPerottoQuarteroni2015}
{Taddei, T.}, {Perotto, S.}, and {Quarteroni, A.}
\newblock Reduced basis techniques for nonlinear conservation laws.
\newblock {\em ESAIM: M2AN}, 49(3):787--814, 2015.

\bibitem{TryoenMaitreErn2012}
J.~Tryoen, O.~Le Ma\^{i}tre, and A.~Ern.
\newblock Adaptive anisotropic spectral stochastic methods for uncertain scalar
  conservation laws.
\newblock {\em SIAM Journal on Scientific Computing}, 34(5):A2459--A2481, 2012.

\bibitem{Welper2015}
G.~Welper.
\newblock Interpolation of functions with parameter dependent jumps by
  transformed snapshots.
\newblock {\em SIAM Journal on Scientific Computing}, 39(4):A1225--A1250, 2017.

\bibitem{XiuKarniadakis2002}
Dongbin Xiu and George~Em Karniadakis.
\newblock The wiener--askey polynomial chaos for stochastic differential
  equations.
\newblock {\em SIAM Journal on Scientific Computing}, 24(2):619--644, 2002.

\bibitem{YanoPateraUrban2014}
Masayuki Yano, Anthony~T. Patera, and Karsten Urban.
\newblock A space-time hp-interpolation-based certified reduced basis method
  for burgers' equation.
\newblock {\em Mathematical Models and Methods in Applied Sciences},
  24(09):1903--1935, 2014.

\end{thebibliography}

\end{document}